\documentclass[11pt,reqno]{amsart}

\usepackage{amssymb}
\usepackage{mathdots} %
\usepackage{bbm} %
\usepackage[shortlabels]{enumitem} %
\usepackage{geometry}
\usepackage{adjustbox} %
\usepackage{verbatim} %

\usepackage{hyperref}
\usepackage{xcolor}
\definecolor{darkgreen}{rgb}{0,0.45,0}
\hypersetup{colorlinks,urlcolor=blue,citecolor=darkgreen,linkcolor=blue,linktocpage=false}
\usepackage[abbrev]{amsrefs}
\usepackage[all,color]{xy}
\newdir{ >}{{}*!/-8pt/@{>}} %
\SilentMatrices
\usepackage{float} %
\usepackage{pgfplots}
\pgfplotsset{compat=newest}

\usepackage{tikz}
\usetikzlibrary{cd}
\usetikzlibrary{decorations.markings}

\tikzset{
 ->-/.style = {
  decoration = {markings, mark=at position #1 with {\arrow[scale=1.3]{>}}},
  postaction = {decorate}
 }
}

\usepackage{ifpdf}
\ifpdf
\else
\usepackage{breakurl}
\fi

\usepackage{mathtools} %

\numberwithin{equation}{section}

\theoremstyle{plain}
\newtheorem{thm}[equation]{Theorem}
\newtheorem*{thm*}{Theorem} %
\newtheorem{prop}[equation]{Proposition}
\newtheorem{lem}[equation]{Lemma}
\newtheorem{cor}[equation]{Corollary}
\newtheorem*{cor*}{Corollary} %

\theoremstyle{definition}
\newtheorem{defn}[equation]{Definition}
\newtheorem{ex}[equation]{Example}
\newtheorem{nota}[equation]{Notation}

\theoremstyle{remark}

\newtheorem{ques}[equation]{Question}
\newtheorem{rem}[equation]{Remark}

\makeatletter
\newcommand{\Eqref}[1]{\textup{\tagform@{\ref*{#1}}}}
\makeatother

\newcommand{\anti}{\ar@{}[dr]|{\;\;\;\boxed{\textup{\scriptsize -1}}}} %
\newcommand{\pb}{\ar@{}[dr]|{\mbox{\huge $\lrcorner$}}} %
\newcommand{\po}{\ar@{}[dr]|{\mbox{\huge $\ulcorner$}}} %

\newcommand{\Def}[1]{\textbf{\boldmath{#1}}} %

\newcommand{\Z}{\mathbb{Z}}

\newcommand{\unit}{\mathbbm{1}} %

\newcommand{\bmat}[1]{\left[\!\!\!\begin{array}{#1}} %
\newcommand{\emat}{\end{array}\!\!\right]} %
\newcommand{\al}{\alpha}

\newcommand{\de}{\delta}
\newcommand{\De}{\Delta}
\newcommand{\del}{\partial}
\renewcommand{\epsilon}{\varepsilon}
\newcommand{\ep}{\varepsilon}
\newcommand{\ga}{\gamma}
\newcommand{\Ga}{\Gamma}
\newcommand{\io}{\iota}

\newcommand{\phy}{\varphi}
\newcommand{\Si}{\Sigma}

\newcommand{\inj}{\hookrightarrow}

\newcommand{\Ra}{\Rightarrow} %
\newcommand{\ral}{\xrightarrow} %
\newcommand{\rla}{\rightleftarrows}

\newcommand{\ot}{\otimes}
\newcommand{\sm}{\wedge} %
\newcommand{\x}{\times}

\newcommand{\tild}{\widetilde}
\newcommand{\ul}{\underline}

\newcommand{\Ab}{\mathrm{Ab}}
\newcommand{\Alg}[1]{\mathrm{Alg}_{#1}} %
\newcommand{\Cat}{\mathrm{Cat}}
\newcommand{\CAT}[1]{#1\text{-}\mathrm{Cat}} %
\newcommand{\Ch}{\mathrm{Ch}}

\newcommand{\Mod}[1]{\mathrm{Mod}_{#1}} %
\newcommand{\lMod}[1]{\mathrm{Mod}_{#1}} %
\newcommand{\Set}{\mathrm{Set}}
\newcommand{\Sp}{\mathrm{Sp}} %
\newcommand{\Top}{\mathrm{Top}}
\newcommand{\cA}{\mathcal{A}}
\newcommand{\cC}{\mathcal{C}}
\newcommand{\cD}{\mathcal{D}}
\newcommand{\cO}{\mathcal{O}}
\newcommand{\cU}{\mathcal{U}}
\newcommand{\cV}{\mathcal{V}}
\newcommand{\cW}{\mathcal{W}}

\newcommand{\abs}[1]{\lvert #1 \rvert}
\DeclareMathOperator{\Cyl}{Cyl}

\DeclareMathOperator{\Fun}{Fun} %
\DeclareMathOperator{\Ho}{Ho} %
\DeclareMathOperator{\Hom}{Hom}
\DeclareMathOperator{\HOM}{\underline{Hom}} %

\DeclareMathOperator{\Path}{Path}

\DeclareMathOperator{\Triv}{Triv}

\newcommand{\AW}{\mathrm{AW}} %
\newcommand{\coloc}{\mathrm{coloc}}

\newcommand{\EZ}{\mathrm{EZ}} %
\newcommand{\Fin}{\mathrm{Fin}} %
\newcommand{\id}{\mathrm{id}}
\newcommand{\Id}{\mathrm{Id}}

\newcommand{\loc}{\mathrm{loc}}
\newcommand{\opp}{\mathrm{op}}

\def\noteson{\gdef\note##1{\noindent{\color{blue}[##1]}}}

\noteson

\newcommand{\blue}[1]{{\color{blue}#1}}

\newcommand{\red}[1]{{\color{red}#1}}

\begin{document}

\title{Enriched model categories and the Dold--Kan correspondence} 
\date{\today}

\author{Martin Frankland} 
\address{University of Regina\\
3737 Wascana Parkway\\
Regina, Saskatchewan, S4S 0A2\\
Canada}
\email{Martin.Frankland@uregina.ca}

\author{Arnaud Ngopnang Ngomp\'e}
\address{University of Regina\\
3737 Wascana Parkway\\
Regina, Saskatchewan, S4S 0A2\\
Canada}
\email{Arnaud.NgopnangNgompe@uregina.ca}
\curraddr{\textsc{AIMS Cameroon\\ Ngeme\\
P.O.\ Box 608 Limbe\\ South-West,
Cameroon}}
\email{Arnaud.Ngopnang@aims-cameroon.org}

\begin{abstract}
The monoidal properties of the Dold--Kan correspondence have been studied in homotopy theory, notably by Schwede and Shipley. 
Changing the enrichment of an enriched, tensored, and cotensored category along the Dold--Kan correspondence does not preserve the tensoring nor the cotensoring. 
More generally, what happens to an enriched model category if we change the enrichment along a weak monoidal Quillen pair? 
We prove a change of base theorem that describes which properties are preserved and which are weakened. 
We also provide sources of examples of weak monoidal Quillen pairs, including in equivariant homotopy theory.
\end{abstract}



\maketitle

\tableofcontents

\section{Introduction}

\subsection*{Background}

Seminal work of Quillen on simplicial model categories showed that combining an enrichment in simplicial sets with a model structure provides additional tools, in particular a simpler way to compute derived mapping spaces \cite{Quillen67}*{\S II} \cite{GoerssJ09}*{\S II}. Model categories can be enriched over a more general symmetric monoidal model category $\cV$, leading to the notion of $\cV$-model category \cite{MayP12}*{\S 16} \cite{Riehl14}*{\S 11} \cite{GuillouM20}. Categories enriched in chain complexes (i.e.\ differential graded categories) are commonly used in algebraic geometry and representation theory, while categories enriched in spectra play an important role in stable homotopy theory. 
What if we want to change the base of enrichment $\cV$? Recall that changing the enrichment along a lax monoidal functor $G \colon \cV \to \cW$ sends a $\cV$-enriched category $\cC$ to a $\cW$-enriched category $G_* \cC$ obtained by applying $G$ to each hom-object. There is a change of base theorem for enriched model categories, found for instance in \cite{GuillouM20}*{Proposition~3.8}, sketched in \cite{Dugger06}*{Lemma~A.5}, %
and alluded to in \cite{Riehl14}*{\S 3.8}. %

\begin{thm}
Let $F \colon \cW \rla \cV \colon G$ be a %
Quillen adjunction between symmetric monoidal model categories. Assume that the left adjoint $F$ is strong monoidal (hence the right adjoint $G$ is lax monoidal) and that $F(Q\unit) \to F(\unit) \cong \unit$ is a weak equivalence, where $Q\unit \ral{\sim} \unit$ is a cofibrant replacement of the tensor unit. If $\cC$ is a $\cV$-model category, then $G_* \cC$ is a $\cW$-model category.
\end{thm}

However, some interesting Quillen adjunctions do not have a strong monoidal left adjoint. %
For a commutative ring $R$, the Dold--Kan correspondence
\[
\xymatrix{
s\Mod{R} \ar@/^0.5pc/[r]^-{N} \ar@{}[r]|-{\sim} & \Ch_{\geq 0}(R) \ar@/^0.5pc/[l]^-{\Ga} \\
}
\]
between simplicial $R$-modules and non-negatively graded chain complexes of $R$-modules is lax monoidal in both directions, but neither direction is strong. In \cite{SchwedeS03equ}, Schwede and Shipley made the key observation that Dold--Kan is ``strong monoidal up to homotopy'', coining the term \emph{weak monoidal Quillen pair} for the relevant properties. Their focus was to investigate the homotopy theory of monoids in monoidal model categories. In particular, they obtained a Quillen equivalence between the monoids on both sides of Dold--Kan:
\[
\xymatrix{
s\Alg{R} \ar@/^0.5pc/[r] \ar@{}[r]|-{\sim} & \mathrm{DGA}_{R,\geq 0}, \ar@/^0.5pc/[l] \\
}
\]
simplicial $R$-algebras and differential graded $R$-algebras in non-negative degrees \cite{SchwedeS03equ}*{Theorem~1.1}. Weak monoidal Quillen pairs have been used in stable homotopy theory \cite{Shipley07} and the homotopy theory of operads and their algebras \cite{Muro14} \cite{PavlovS18} %
\cite{WhiteY19}.

\subsection*{Main results}

In this paper, we tackle this question: What happens to enriched model categories when we change the enrichment along a \emph{weak} monoidal Quillen pair? Our proposed answer is that most of the structure is preserved, except the tensoring and cotensoring, which are weakened to an ``up to homotopy'' version. Here is our main result (see Theorem~\ref{thm:maintheo}).

\begin{thm*}
Let $F \colon \cW \rla \cV \colon G$ be a weak monoidal Quillen pair such that the map \linebreak[4] \mbox{$\ep \colon F(\unit_{\cW}) \ral{\cong} \unit_{\cV}$} is an isomorphism. If $\cC$ is a weak $\cV$-model category, then $G_*\cC$ is a weak $\cW$-model category.
\end{thm*}

See Definition~\ref{def:wvmod} for the notion of weak $\cV$-model category. The novelty is that the tensoring and cotensoring have been weakened, in the sense of Definition~\ref{def:wtwc}.

In Section~\ref{sec:App}, we work out sample applications with the Dold--Kan correspondence. We recover a known homotopy equivalence between two simplicial enrichments on chain complexes and vice versa. %

\subsection*{Related work}

In this paper, we do not use the symmetry properties of the Dold--Kan correspondence. The applications to homotopy theory were investigated in detail by Richter \cite{Richter03}, and also Mandell \cite{Mandell03}*{Theorem~1.3}. 

In \cite{Tabuada10}, Tabuada showed that change of enrichment along the Dold--Kan correspondence induces a Quillen equivalence
\[
\xymatrix{
\CAT{s\Mod{R}} \ar@/^0.5pc/[r]^-{N_*} \ar@{}[r]|-{\sim} & \CAT{\Ch_{\geq 0}(R)} \ar@/^0.5pc/[l]^-{\Ga_*} \\
}
\]
between categories enriched in simplicial $R$-modules and categories enriched in non-negatively graded chain complexes of $R$-modules, where both sides are endowed with a Bergner-style model structure. In our paper, we do not consider a model structure on the category $\CAT{\cV}$ of all \mbox{$\cV$-categories} and $\cV$-functors. Rather, we work with individual $\cV$-categories equipped with a compatible model structure.

Throughout the paper, we view ``$\cV$-enriched'' as an adjective, as additional structure on the underlying category. As such, we work with model structures on underlying categories and we focus on a change of enrichment that preserves the underlying category (see Lemma~\ref{lem:sameunderlying}). To treat the more general case that changes the underlying category, it may be helpful to view ``$\cV$-enriched category'' as a noun and adopt the notion of enriched model category with enriched weak factorization systems \cite{Riehl14}*{\S 13}.

Also, we do not address the relationship between model-categorical constructions and $\infty$-categories. Given a (symmetric) monoidal model category $\cV$, the underlying symmetric monoidal $\infty$-category $\Ho_{\infty}(\cV)$ is described for instance in \cite{Lurie17}*{Example~4.1.7.6}. We surmise that the following is true, which is leave as a question to the reader.

\begin{ques}
Given a weak $\cV$-model category $\cC$ (in the sense of Definition~\ref{def:wvmod}), is its underlying $\infty$-category $\Ho_{\infty}(\cC)$ enriched, tensored, and cotensored over $\Ho_{\infty}(\cV)$?
\end{ques}

Different approaches to enriched and (co)tensored $\infty$-categories can be found in \cite{GepnerH15}, \cite{Lurie17}*{\S 4.2}, and \cite{Heine23}.

\subsection*{Organization}

In Section~\ref{sec:Ptc}, we give the conditions under which the tensoring and cotensoring structures are preserved (Proposition~\ref{pr:coten-strong}). In Section~\ref{sec:wea}, we introduce the notion of weak enriched adjunction and we show that any weak monoidal Quillen adjunction lifts to a weak enriched adjunction (Proposition~\ref{pr:prob9}). In Section~\ref{sec:swtwc}, we introduce the notions of weak tensoring and weak cotensoring, and we show that they are preserved by a weak monoidal Quillen adjunction (Proposition~\ref{pr:PropW_ten_coPreserv}). In Section~\ref{sec:ua}, we introduce equivalent formulations of the unit axiom and we prove that some implications between them still hold when the tensoring is replaced by its weak version (Proposition~\ref{pr:ImplicationsUnit}). We also show that a weak (co)tensoring induces a (co)tensoring on the homotopy category (Proposition~\ref{pr:PropHoenrich}). In Section~\ref{sec:Pems}, we introduce the notion of weak $\cV$-model category (Definition~\ref{def:wvmod}) and we show that this is preserved by a change of enrichment along a weak monoidal Quillen adjunction (Theorem~\ref{thm:maintheo}). In Section~\ref{sec:App}, we apply some of the above results in the case of the Dold--Kan correspondence. %
In Part~II, %
we describe ways to construct weak monoidal Quillen pairs using diagram categories (Proposition~\ref{pr:DiagramCat}), equivariant homotopy (Proposition~\ref{pr:GObjects} and Example~\ref{ex:EquivariantDoldKan}), or Bousfield localizations (Corollary~\ref{cor:MonoidalBousfield}).

\subsection*{Acknowledgments}

This work is based on the second author's doctoral thesis \cite{Ngopnang24enr}, supervised by the first author and Don Stanley. 
We thank Kate Ponto for many helpful suggestions.  
We also thank Chris Kapulkin, Nick Rozenblyum, Alejandro Saiz Millán, Stefan Schwede, Marc Stephan, David White, Sinan Yalin, and Donald Yau for helpful discussions. 

Frankland acknowledges the support of the Natural Sciences and Engineering Research Council of Canada (NSERC), Discovery Grant RGPIN-2019-06082, as well as support from the Fields Institute and the Max-Planck-Institut für Mathematik, report number MPIM-Bonn-2026. 
Cette recherche a \'et\'e financ\'ee par le Conseil de recherches en sciences naturelles et en g\'enie du Canada (CRSNG), subvention D\'ecouverte RGPIN-2019-06082. 
Ngopnang acknowledges the support of the Faculty of Graduate Studies and Research at the University of Regina. %
We gratefully acknowledge the support of the Pacific Institute for the Mathematical Sciences, report identifier PIMS-20250819-CRG41.

\section{Preliminaries}

\subsection{Monoidal categories and enriched categories}\label{sec:mcec}

We recall some background material on monoidal categories and enriched categories. We mostly follow \cite{Riehl14}*{\S 3} and \cite{Borceux94v2}*{\S 6}; see also \cite{Kelly05}. 
A $\cV$-enriched category will also be called a $\cV$-category.

\begin{defn}%
	A symmetric monoidal category $\cV$ is \Def{closed} if for all object $v \in \cV$, the tensor product functor $- \otimes v: \cV \to \cV$ has a right adjoint functor $[v,-] : \cV \to \cV$, 
 that is, for any $u,v,w \in \cV$, we have a natural bijection
 \begin{equation}%
     \cV(u\otimes v, w)\cong \cV(u,[v,w]),
 \end{equation}
 natural in all arguments. The object $[v,w]$ is called the \Def{internal hom} of $v$ and $w$, also denoted $\underline{\cV}(v,w):=[v,w]$ or $w^v$.
\end{defn}

\begin{prop}\label{pr:Internalize}\cite{Riehl14}*{Remark~3.3.9} 
 Given a closed symmetric monoidal category $\cV$, the internal hom satisfies an isomorphism
\begin{equation*}
        \underline{\cV}(u \otimes v, w) \cong \underline{\cV}(u,\underline{\cV}(v,w))
\end{equation*}
in $\cV$, natural in all arguments. 
This isomorphism is called the \Def{enriched tensor-hom adjunction}.
\end{prop}

\begin{defn}
	For two monoidal categories $(\cV,\otimes,\unit_{\cV})$ and $(\cW,\otimes,\unit_{\cW})$, a \Def{lax monoidal functor} is a functor $G:\cV\to \cW$ together with
	\begin{itemize}
		\item a morphism $\eta: \unit_{\cW}\to G(\unit_{\cV})$, called the lax monoidal unit,
		\item a natural transformation $\mu$, given by $\mu_{x,y}:G(x)\otimes G(y)\to G(x\otimes y)$ for any objects $x,y\in \cV$, called the lax monoidal transformation,
    \end{itemize}
    satisfying 
	associativity and unitality. 
	A lax monoidal functor $G$ is said to be \Def{strong monoidal} if the transformations $\mu$ and $\eta$ are isomorphisms. 
	
	An \Def{oplax monoidal functor} $F : \cW \to \cV$ %
	is a functor together with a morphism $\ep \colon F(\unit_{\cW}) \to \unit_{\cV}$ and a natural transformation $\de_{x,y} \colon F(x \ot y) \to F(x) \ot F(y)$ satisfying coassociativity and counitality.
\end{defn}

\begin{nota}
For a monoidal category $\cV$, let $\CAT{\cV}$ denote the category of $\cV$-categories and $\cV$-functors between them. We can also view $\CAT{\cV}$ as a $2$-category, with the $\cV$-natural transformations as $2$-morphisms.
\end{nota}

The following construction is called \Def{change of base} or \Def{change of enrichment} \cite{Borceux94v2}*{Proposition~6.4.3}.

\begin{prop}%
Let $G \colon \cV \to \cW$ be a lax monoidal functor between monoidal categories.
\begin{enumerate}
	\item Any $\cV$-category $\cC$ has an associated $\cW$-category, denoted $G_*\cC$, with the same objects, and hom-objects given by $\underline{G_*\cC}(x,y) := G (\underline{\cC}(x,y))$.
	\item The assignment $\cC \mapsto G_* \cC$ forms a $2$-functor $G_* \colon \CAT{\cV} \to \CAT{\cW}$.
\end{enumerate}
\end{prop}

\begin{lem}\label{lem:Compochange}
    The change of enrichment along a composite $(GH)_*$ is the composite of change of enrichments $G_*H_*$.
\end{lem}

\begin{lem}\cite{Borceux94v2}*{Proposition~6.4.2}
	The underlying set functor $U = \cV(\unit,-) \colon \cV \to \Set$ is lax monoidal.
\end{lem}

\begin{defn}
	Given a $\cV$-category $\cC$, its \Def{underlying category} $\cC_0 = U_* \cC$ is the category obtained by change of enrichment along the underlying set functor $U \colon \cV \to \Set$. In other words, $\cC_0$ 
	has the same objects as $\cC$ and has hom-sets $\cC(x,y) := \cV(\unit, \underline{\cC}(x,y))$, for any $x,y \in \cC$. 
\end{defn}

The following statement is an instance of doctrinal adjunction \cite{Kelly74}.

\begin{lem}\label{lem:Doctrinal}
Let $F \colon \cW \rla \cV \colon G$ be an adjunction between monoidal categories. An oplax monoidal structure on the left adjoint $F$ corresponds to a lax monoidal structure on the right adjoint $G$. Given an oplax monoidal structure on $F$, the corresponding lax monoidal structure on $G$ is defined by:
\[
\xymatrix{
Gx \ot Gy \ar[d]_{\eta} \ar[r]^-{\mu}_{\blue{\mathrm{def}}} & G(x \ot y) \\
GF (Gx \ot Gy) \ar[r]^-{G(\de)} & G(FGx \ot FGy) \ar[u]_{G(\ep \ot \ep)} \\ 
}
\qquad \qquad 
\xymatrix{
\unit_{\cW} \ar[d]_{\eta} \ar[r]^-{\eta}_-{\blue{\mathrm{def}}} & G(\unit_{\cV}). \\
GF(\unit_{\cW}) \ar[ur]_-{G(\ep)} & \\
}
\]
Here $\eta \colon \id_{\cW} \Ra GF$ and $\ep \colon FG \Ra \id_{\cV}$ denote the unit and counit of the adjunction $F \dashv G$, $\eta \colon \unit_{\cW} \to G(\unit_{\cV})$ denotes the lax unit of $G$, and $\ep \colon F(\unit_{\cW}) \to \unit_{\cV}$ denotes the oplax counit of $F$.
\end{lem}

The enriched Yoneda lemma will be an important tool in some of the later proofs.

\begin{prop}[Enriched Yoneda Lemma]
    Given a $\cV$-functor $F \colon \cC \to \cV$ and an object $x \in \cC$, the set of $\cV$-natural transformations $\alpha \colon \underline{\cC}(x,-) \to F$ is in natural bijection with the set of elements of $F(x) \cong \cV(\unit, F(x))$, that is, the set of morphisms $\unit \to F(x)$, obtained by the composition $\unit \xrightarrow{1_x} \underline{\cC}(x,x) \xrightarrow{\alpha_x} F(x)$.
\end{prop}

%
\begin{comment}
The (Enriched) Yoneda Lemma yields the following.

\begin{lemma} 
	Let $\cC$ be a $\cV$-category. The following statements are equivalent.
	\begin{enumerate}
		\item The objects $x,y\in \cC$ are isomorphic as objects of $\cC$.
		\item The representable functors $\xymatrix{\cC(x,-),\cC(y,-)\colon \cC  \ar@<1ex>[r] \ar@<-.6ex>[r] & \mathbf{Set}}$ are naturally isomorphic.
		\item The underlying functors of the representable functors $\xymatrix{ \underline{\cC}(x,-), \underline{\cC}(y,-)\colon \cC \ar@<1ex>[r] \ar@<-.6ex>[r] & \cV }$ are naturally isomorphic.
		\item The representable $\cV$-functors $\xymatrix{ \underline{\cC}(x,-), \underline{\cC}(y,-)\colon \cC \ar@<1ex>[r] \ar@<-.6ex>[r] & \cV }$ are $\cV$-isomorphic.
	\end{enumerate}
\end{lemma}

\begin{definition} 
	A \Def{$\cV$-equivalence of categories} is given by a $\cV$-functor \\ $F\colon\cC\to \cD$ satisfying the following properties.
	\begin{itemize}
		\item \Def{Essentially surjective}, i.e., every  $d\in\cD$ is isomorphic (in $\cD_0$) to some $Fc$.
		\item $\cV$\Def{-Fully faithful}, i.e., for each $x,y\in \cC$, the map $F_{x,y}\colon\underline{\cC}(x,y)\to \underline{\cD}(Fx,Fy)$ is an isomorphism in $\cV$.
	\end{itemize}
\end{definition}
\end{comment}
%

\begin{defn}
	A \Def{$\cV$-adjunction} consists of $\cV$-functors $F \colon \cC \to \cD$ and \mbox{$G \colon \cD \to \cC$} together with
	\begin{itemize}
		\item $\cV$-natural isomorphisms $\underline{\cD}(Fc,d) \cong \underline{\cC}(c,Gd)$ in $\cV$, for any $c \in \cC$ and $d \in \cD$, or equivalently
		\item $\cV$-natural transformations $\eta \colon \mathrm{id}_{\cC} \Ra GF$ (unit) and $\varepsilon \colon FG \Ra \mathrm{id}_{\cD}$ (counit) satisfying the triangle identities 
        \begin{center}
		\begin{tabular}{ccc}
			$\xymatrix{
			G \ar[dr]_-{\mathrm{id}_G} \ar[r]^-{\eta G} & GFG \ar[d]^-{G\epsilon} \\
			& G,
			}$
			& $\qquad\qquad$ &
			$\xymatrix{
			F \ar[dr]_-{\mathrm{id}_F} \ar[r]^-{F\eta} & FGF \ar[d]^-{\epsilon F} \\
			& F.
			}$
		\end{tabular}
	\end{center}
	\end{itemize}
\end{defn}

\subsection{Tensored and cotensored categories}\label{sec:tcc}

Background material on tensored and cotensored categories can be found in \cite{Kelly05}*{\S 1}, \cite{Borceux94v2}*{\S 6.5} and \cite{Riehl14}*{\S 3.7}.

\begin{defn}
	\begin{enumerate}
		\item A $\cV$-category $\cC$ is \Def{tensored} over $\cV$ if for any $v \in \cV$ and $x \in \cC$, there is an object $x \otimes v \in \cC$ together with isomorphisms in $\cV$ 
		\begin{equation*}
			\underline{\cC}(x \otimes v, y) \cong \underline{\cV}(v,\underline{\cC}(x,y)), \quad \text{$\cV$-natural in } y \in \cC.
		\end{equation*}
		\item A $\cV$-category $\cC$ is \Def{cotensored} over $\cV$ if for any $v \in \cV$ and $y \in \cC$, there is an object $y^v \in \cC$ together with isomorphisms in $\cV$ 
		\begin{equation*}
			\underline{\cC}(x, y^v) \cong \underline{\cV}(v,\underline{\cC}(x,y)), \quad \text{$\cV$-natural in } x \in \cC.
		\end{equation*}
	\end{enumerate}
\end{defn}

\begin{rem}
    Given a $\cV$-category $\cC$ tensored over $\cV$, for each object $x \in \cC$, the $\cV$-functor $x \otimes - \colon \ul{\cV} \to \ul{\cC}$ is left $\cV$-adjoint to $\underline{\cC}(x,-) \colon \ul{\cC} \to \ul{\cV}$. This follows from the enriched Yoneda lemma.
\end{rem}

\begin{rem}\label{rem:UnenrichedTensorHom}
In contrast to Proposition~\ref{pr:Internalize}, having natural isomorphisms of hom-sets 
\[
\cC(x \otimes v, y) \cong \cV(v, \underline{\cC}(x,y))
\]
is \emph{not} enough to produce a tensoring of $\cC$ over $\cV$. The issue is enriching the unenriched adjunction $x \otimes - \colon \cV \rla \cC \colon \underline{\cC}(x,-)$, cf.\ \cite{Riehl14}*{Proposition~3.7.10}.
\end{rem}

\begin{ex}\label{ex:Modules}
For a commutative ring $R$, the category of $R$-modules $\Mod{R}$ endowed with the usual tensor product $\ot_R$ over $R$ and hom modules $\Hom_R(M,N)$ is a closed symmetric monoidal category. For an arbitrary ring $R$, the category of left $R$-modules $\lMod{R}$ is enriched, tensored, and cotensored over the category $\Ab$ of abelian groups.
\end{ex}

\begin{lem}
	Let $(\cV,\otimes,\unit)$ be a closed symmetric monoidal category and suppose $\cC$ is a tensored $\cV$-category. Then the tensoring is unital and associative, i.e., there exist natural isomorphisms
	\begin{equation*}
		x \otimes \unit \cong x, \quad x \otimes (u \otimes v) \cong (x \otimes u) \otimes v , \quad \text{for all } u,v \in \cV, x \in \cC.
	\end{equation*}
\end{lem}

\begin{prop}\label{pr:etcpreserv} \cite{Riehl14}*{Theorem~3.7.11}
    Suppose we have an adjunction $F \colon \cW \rla \cV \colon G$ between closed symmetric monoidal categories such that the left adjoint $F$ is strong monoidal. Then for any tensored and cotensored $\cV$-category $\cC$, the $\cW$-category $G_*\cC$ becomes canonically tensored and cotensored over $\cW$, given respectively by
    \begin{equation*}
        x \otimes w := x \otimes Fw \quad \text{and} \quad x^w := x^{Fw}, \ \text{for any } x \in \cC, w \in \cW.
    \end{equation*}
\end{prop}

\begin{cor}
    The strong monoidal adjunction $F \dashv G$ of Proposition~\ref{pr:etcpreserv} is a $\cW$-adjunction with respect to the induced $\cW$-category structure on $\cV$, i.e., on $G_*\cV$.
\end{cor}

In the following proposition, we see how to transport the tensoring and the cotensoring along an equivalence of categories.

\begin{prop}\label{pr:EquivCoTen}
    Let $\cD$ be a $\cV$-enriched category tensored and cotensored over a closed symmetric monoidal category $\cV$, and an equivalence of categories $F \colon \cC \rla \cD \colon G$.
    \begin{enumerate}
        \item We can transport the tensoring from $\cD$ to $\cC$ by
        \begin{equation*}
            c \otimes v := G(Fc\otimes v), \ \text{for all } c \in \cC  \text{ and } v \in \cV.
        \end{equation*}
        \item We can transport the cotensoring from $\cD$ to $\cC$ by
        \begin{equation*}
            c^v := G((Fc)^v), \ \text{for all } c \in \cC \text{ and }  v \in \cV.
        \end{equation*}
    \end{enumerate}
\end{prop}

\subsection{Properties of the Dold--Kan correspondence}\label{sec:mpDKc}

Here we recall the monoidal structures on chain complexes and simplicial modules, cf.\ \cite{MayP12}*{Example~16.3.1}. 
For background on simplicial objects, see \cite{Weibel94}*{\S 8.1} or \cite{GoerssJ09}*{\S I.1}.

\begin{nota}
    Let $\Delta$ denote the simplex category, with objects $[n] = \{ 0 < 1 < \cdots < n \}$ and morphisms the order-preserving functions $[m] \to [n]$. For a category $\cC$, denote the category of simplicial (resp.\ cosimplicial) objects in $\cC$ by
    \[
    s\cC = \Fun(\De^{\opp},\cC) = \cC^{\De^{\opp}} \quad \text{and} \quad c\cC = \Fun(\De,\cC) = \cC^{\De}.
    \]
\end{nota}

\begin{defn}\label{def:tenCh} 
Let $R$ be a commutative ring, and $X,Y \in \Ch(R)$ chain complexes of $R$-modules.
\begin{enumerate}
	\item %
		The \Def{tensor product} $X \ot Y$ of chain complexes is defined by:
	\begin{equation*}
		(X \otimes Y)_n := \bigoplus_{i \in \Z} X_i \otimes_R Y_{n-i}, \quad \text{with} \quad d(x \otimes y) :=  d(x) \otimes y + (-1)^{|x|} x \otimes d(y).
	\end{equation*}
	\item The \Def{hom complex} is defined by:
	\begin{equation*}
		\HOM_{\mathrm{Ch}(R)}(X,Y)_n := \prod_{i\in \Z} \Hom_R(X_i,Y_{i+n}), \quad \text{with} \quad (df)(x) := d(f(x)) - (-1)^{|f|} f(d(x)).
	\end{equation*}
	\item The \Def{good truncation} $\tau_{\geq 0} C$ is defined by: 
	\begin{equation*}
		(\tau_{\geq 0}C)_n := \begin{cases}
			C_n &\text{if}\ n\geq 1 \\
			\ker(d_0) &\text{if}\ n=0 \\
			0 &\text{otherwise.} \\
		\end{cases}
	\end{equation*}
\end{enumerate}
\end{defn}

\begin{lem}\label{lem:unboundmonoidal}
Let $R$ be a commutative ring.
	\begin{enumerate}
		\item The category $\Ch(R)$ of unbounded chain complexes of $R$-modules endowed with the tensor product $\otimes$ of chain complexes and the hom complex $\HOM_{\Ch(R)}$ %
		is a closed symmetric monoidal category. The tensor unit $R[0]$ is the chain complex with $R$ concentrated in degree $0$.
		\item The category $\Ch_{\geq 0}(R)$ endowed with the tensor product $\otimes$ of chain complexes and the truncated hom complex $\HOM_{\Ch_{\geq 0}(R)} := \tau_{\geq 0}(\HOM_{\Ch(R)})$ %
		is a closed symmetric monoidal category.
		\item The category $s\Mod{R}$ of simplicial $R$-modules is a closed symmetric monoidal category. The tensor product $A \ot B$ is the degreewise tensor product $(A \otimes B)_n := A_n \otimes_R B_n$. %
		The tensor unit is the constant simplicial $R$-module $c(R)$. %
		The internal hom of simplicial $R$-modules $A$ and $B$ is the simplicial $R$-module $\HOM_{s\Mod{R}}(A,B)$ given in degree $n$ by the $R$-module
		\begin{equation*}
			\HOM_{s\Mod{R}}(A, B)_n := \Hom_{s\Mod{R}}(A \otimes R\Delta^n, B).
		\end{equation*}
	\end{enumerate}
\end{lem}

\begin{lem}\label{lem:NonCommutative}
Let $R$ be a ring (not necessarily commutative).
	\begin{enumerate}
		\item The category $\Ch(R)$ is enriched, tensored, and cotensored over the category $\Ch(\Z)$ of unbounded chain complexes of abelian groups.
		\item The category $\Ch_{\geq 0}(R)$ is enriched, tensored, and cotensored over the category $\Ch_{\geq 0}(\Z)$ of non-negatively graded chain complexes of abelian groups.
		\item The category $s\lMod{R}$ of simplicial left $R$-modules is enriched, tensored, and cotensored over the category $s\Ab$ of simplicial abelian groups.
	\end{enumerate}
\end{lem}

\begin{lem}\label{lem:GoodTruncation}
Let $R$ be a ring. 
	\begin{enumerate}
		\item The good truncation functor $\tau_{\geq 0} \colon \mathrm{Ch}(R) \to \mathrm{Ch}_{\geq 0}(R)$ is right adjoint to the inclusion functor $\io \colon \Ch_{\geq 0}(R) \inj \Ch(R)$.
		\item If moreover $R$ is commutative, then the inclusion functor $\io$ is strong monoidal, hence its right adjoint $\tau_{\geq 0}$ is lax monoidal.
	\end{enumerate}
\end{lem}

Given a ring $R$, 
the normalization functor $N \colon s\Mod{R} \to \Ch_{\geq 0}(R)$ and the denormalization functor $\Gamma \colon \Ch_{\geq 0}(R) \to s\Mod{R}$ form the Dold--Kan correspondence, described for instance in \cite{GoerssJ09}*{\S III.2} or \cite{Weibel94}*{\S 8.4}. The monoidal properties of Dold--Kan are described in detail in \cite{SchwedeS03equ}*{\S 2}. 
We recall the salient facts in an omnibus theorem.

\begin{thm}[Dold--Kan correspondence]\label{thm:DoldKan}
Let $R$ be a ring.
\begin{enumerate}
\item \label{item:Equivalence} The normalization and denormalization $N \colon s\Mod{R} \rla \Ch_{\geq 0}(R) \colon \Ga$ form an adjoint equivalence of categories.
\item \label{item:Homotopy} Two maps of simplicial $R$-modules $f,g \colon A \to B$ are homotopic if and only if their normalizations $Nf, Ng \colon N(A) \to N(B)$ are chain homotopic.
\item The homotopy groups of a simplicial $R$-module $A$ correspond to the homology groups of its normalization: $\pi_n(A) \cong H_n(NA)$.
\end{enumerate}
Now assume that $R$ is commutative, so that both $s\Mod{R}$ and $\Ch_{\geq 0}(R)$ are symmetric monoidal categories.
\begin{enumerate}
\setcounter{enumi}{3}
\item Normalization is lax monoidal via the Eilenberg--Zilber map (also called shuffle map)
    \begin{equation*}
        \EZ \colon N(A) \otimes N(B) \to N(A\otimes B), \quad \text{for any } A,B \in s\Mod{R}.
    \end{equation*}
\item Normalization is oplax monoidal via the Alexander--Whitney map
    \begin{equation*}
        \AW \colon N(A \otimes B) \to N(A) \otimes N(B), \quad \text{for any } A,B \in s\Mod{R}.
    \end{equation*}
\item The composite %
\[
\xymatrix{
N(A) \ot N(B) \ar[r]^-{\EZ} & N(A \ot B) \ar[r]^-{\AW} & N(A) \ot N(B) \\
}
\]
is the identity, and the composite $\EZ \circ \AW$ is naturally chain homotopic to the identity. Hence the chain complex $N(A) \otimes N(B)$ is a deformation retract of $N(A \otimes B)$.
\item The natural isomorphism of chain complexes $\ep \colon N \Ga(C) \ral{\cong} C$ is a monoidal natural transformation.
\item Any choice of natural isomorphism of simplicial $R$-modules $\eta \colon A \ral{\cong} \Ga N(A)$ is \emph{not} a monoidal transformation. 
\end{enumerate}
\end{thm}

For the record, 
Lemmas~\ref{lem:unboundmonoidal} and \ref{lem:NonCommutative} 
can be extended to 
more general abelian categories than 
the category $\Mod{R}$ of $R$-modules. %

\begin{prop}\label{pr:Prop1}
	Let $\cA$ be a bicomplete closed symmetric monoidal abelian category. 
The categories $\Ch(\cA)$, $\Ch_{\geq 0}(\cA)$, and $s\cA$ are closed symmetric monoidal categories.
\end{prop}

\begin{prop}\label{pr:Prop2}
	Let $\cA$ be a bicomplete abelian category.
	\begin{enumerate}
		\item $\Ch(\cA)$ is enriched, tensored, and cotensored over $\Ch(\Z)$.
		\item $\Ch_{\geq 0}(\cA)$ is enriched, tensored, and cotensored over $\Ch_{\geq 0}(\Z)$.
		\item $s\cA$ is enriched, tensored, and cotensored over $s\Ab$.
	\end{enumerate}
\end{prop}

\subsection{Model categories}

Background on model categories can be found in \cite{Hovey99}*{\S 1}, \cite{Riehl14}*{\S 2}, \cite{MayP12}*{\S 14}, or \cite{Hirschhorn03}*{\S 7--9}. 
As a matter of terminology, we assume that a model category has all small limits and colimits, but we do not assume that the factorizations are functorial.

For a model category $\cC$, denote by $\gamma \colon \cC \to \Ho(\cC)$ the localization functor, that is, the functor that inverts the weak equivalences in $\cC$. The hom-set $[x,y]$ in the homotopy category $\Ho(\cC)$ can be computed as
\[
[x,y] \cong \cC(x,y)/\!\sim
\]
where $\sim$ is the homotopy relation on the hom-set $\cC(x,y)$ \cite{Hovey99}*{\S 1.2}. We will use the following fact later.

\begin{prop}\label{pr:Saturation}
    Let $\cC$ be a model category and $f \colon x \to y$ a map in $\cC$. The following conditions on $f$ are equivalent.
    \begin{enumerate}
        \item $f$ is a weak equivalence.
        \item $\ga(f)$ is an isomorphism in $\Ho(\cC)$.
        \item For every fibrant object $z \in \cC$, the induced map on hom-sets $f^* \colon [y,z] \to [x,z]$ is a bijection.
        \item For every cofibrant object $a \in \cC$, the induced map on hom-sets $f_* \colon [a,x] \to [a,y]$ is a bijection.        
    \end{enumerate}
\end{prop}

\begin{proof}
    The equivalence (1) $\iff$ (2) is the saturation property \cite{Hovey99}*{Theorem~1.2.10}. The equivalences (3) $\iff$ (2) $\iff$ (4) follow from the Yoneda embedding.
\end{proof}

\subsection{Monoidal model categories and enriched model categories}\label{sec:mmcemc}

Most of the material we review in this section can be found in \cite{Hovey99}*{\S 4}, \cite{MayP12}*{\S 16}, and \cite{Riehl14}*{\S 11}. 

\begin{defn} 
A \Def{monoidal model category} is a model category $\cV$ equipped with the structure of a closed %
monoidal category $(\cV,\otimes,\unit)$ satisfying the following two compatibility conditions.
	\begin{itemize}
		\item[(i)] \textbf{Pushout-product axiom:} For every pair of cofibrations $i \colon a \to b$ and  
		$k \colon x \to y$ in $\cV$, their pushout-product %
		\[
        i\square k \colon (a \otimes y) \amalg_{a\otimes x} (b\otimes x) \to b \otimes y
		\]
		is itself a cofibration in $\cV$, which moreover is acyclic if $i$ or $k$ is.
		\item[(ii)] \textbf{Unit axiom:} For every cofibrant object $x \in \cV$ and every cofibrant replacement of the tensor unit $q\colon Q\unit \to \unit$, %
		the resulting morphism $\xymatrix{x\otimes Q\unit \ar[r]^-{x\otimes q} & x\otimes \unit \ar[r]^-{\cong} & x}$
		is a weak equivalence.
	\end{itemize}
A \Def{symmetric monoidal model category} is a monoidal model category such that $(\cV,\otimes,\unit)$ has the structure of a closed symmetric monoidal category.
\end{defn}

When using a monoidal category as base of enrichment, it is convenient to assume symmetry, which is why we focus on symmetric monoidal model categories in this paper. 

\begin{lem} \cite{MayP12}*{Lemma~16.4.5} 
    The pushout-product axiom is equivalent to the following axiom.
    
    \noindent \textbf{Pullback-power axiom}: For a cofibration $i \colon a \to b$ and a fibration $p \colon x \to y$ in $\cV$, their pullback-power %
	\[    
    (i^*,p_*) \colon [b,x] \to [a,x] \times_{[a,y]} [b,y]
    \]
    is a fibration in $\cV$, which moreover is acyclic if $i$ or $p$ is.
\end{lem}

\begin{ex}\label{ex:MonoidalModelCats}
Here are examples of symmetric monoidal model categories.
\begin{enumerate}
	\item The category $s\Set$ of simplicial sets with the Cartesian product $\x$ and the Kan--Quillen model structure \cite{Hovey99}*{Proposition~4.2.8}. 
	\item The category $s\Set_*$ of pointed simplicial sets with the smash product $\sm$ and the Kan--Quillen model structure \cite{Hovey99}*{Corollary~4.2.10}. 
	\item Any category of convenient topological spaces $\Top$ (for instance $k$-spaces) with the Cartesian product $\x$ and the Serre model structure, also called Quillen or $q$-model structure \cite{Hovey99}*{Proposition~4.2.11} \cite{MayP12}*{Theorem~17.2.2}. This also holds for the Str{\o}m model structure, also called Hurewicz or $h$-model structure \cite{MayP12}*{Theorem~17.1.1}.
	\item \label{item:sMod} For a commutative ring $R$, the category $s\Mod{R}$ %
	with the standard Quillen model structure \cite{SchwedeS03equ}*{\S 4.1}, %
	or the Hurewicz model structure \cite{Ngopnang24hur}*{Proposition~5.4}.
	\item \label{item:Chgeq0} The category $\Ch_{\geq 0}(R)$ %
	with the projective model structure \cite{SchwedeS03equ}*{\S 4.1}, %
	or the Hurewicz model structure \cite{Ngopnang24hur}*{Proposition~3.6}.
	\item The category $\Ch(R)$ %
	with the projective model structure \cite{Hovey99}*{Proposition~4.2.13} \cite{MayP12}*{Theorem~18.4.2}, %
	or the Hurewicz model structure \cite{MayP12}*{Theorem~18.3.1}.
	\item The categories of symmetric spectra or orthogonal spectra with the stable or positive stable model structures \cite{SchwedeS03equ}*{\S 7} \cite{MandellMSS01}*{Theorems~12.1, 14.2} \cite{SchwedeS00}*{\S 5}. %
	\item The category of $S$-modules \cite{SchwedeS03equ}*{\S 7} \cite{SchwedeS00}*{\S 5} \cite{EKMM97}*{\S III.3, VII.4}.
\end{enumerate}
\end{ex}

\begin{defn}\label{def:def-emc} 
Let $\cV$ be a symmetric monoidal model category.
	A \Def{\mbox{$\cV$-model} category} is a $\cV$-category $\cC$, which is tensored and cotensored over $\cV$, with the structure of a model category on the underlying category $\cC_0$ such that the  following two compatibility conditions are satisfied.
	\begin{itemize}
		\item[(i)] \textbf{External pushout-product axiom:} For every pair of cofibrations $i \colon x \to y$ in $\cC_0$ and $k \colon a \to b$ in $\cV$, their pushout-product %
		\[
		i \square k \colon (x\otimes b) \amalg_{x \otimes a} (y \otimes a) \to y \otimes b
		\] 
		is itself a cofibration in $\cC_0$, which moreover is acyclic if $i$ or $k$ is.
		\item[(ii)] \textbf{External unit axiom:} For every cofibrant object $x$ in $\cC_0$ and every cofibrant replacement of the tensor unit $q\colon Q\unit \to \unit$ in $\cV$, the resulting morphism $\xymatrix{x\otimes Q\unit \ar[r]^-{x\otimes q} & x\otimes \unit \ar[r]^-{\cong} & x}$ is a weak equivalence in $\cC_0$.
	\end{itemize}
\end{defn}

\begin{lem} \cite{MayP12}*{Lemma~16.4.5}
    The external pushout-product axiom is equivalent to each of the following axioms.
    \begin{enumerate}
        \item \textbf{External pullback-power axiom}: For a cofibration $i \colon a \to b$ in $\cV$ and a fibration $p \colon x \to y$ in $\cC_0$, their pullback-power %
		\[        
        (i^*,p_*) \colon x^b \to x^a \times_{y^a} y^b 
		\]        
        is a fibration in $\cC_0$, which moreover is acyclic if $i$ or $p$ is.
		\item \textbf{SM7}: For a cofibration $i \colon a \to b$ and a fibration $p \colon x \to y$ both in $\cC_0$, their pullback-power %
    \[
    (i^*,p_*) \colon \underline{\cC}(b,x) \to \underline{\cC}(a,x) \times_{\underline{\cC}(a,y)} \underline{\cC}(b,y) 
    \]
    is a fibration in $\cV$, which moreover is acyclic if $i$ or $p$ is.
    \end{enumerate}
\end{lem}

\begin{ex}
A $s\Set$-model category is a simplicial model category, as introduced by Quillen \cite{Quillen67}*{\S II.2} \cite{GoerssJ09}*{\S II.3}. The external unit axiom was not needed there since the tensor unit $\unit_{s\Set} = \ast$ is cofibrant, as is every object in $s\Set$.
\end{ex}

\begin{ex}
Any symmetric monoidal model category $\cV$ is itself a $\cV$-model category.
\end{ex}

\begin{lem}\label{lem:cofibcoten}
    For any $\cV$-category $\cC$ equipped with a model structure satisfying the external pushout-product axiom, we have the following for any $x,y \in \cC$ and $v \in \cV$.
    \begin{enumerate}
        \item The tensor $x \otimes v \in \cC$ is cofibrant whenever $x$ and $v$ are cofibrant.
        \item The cotensor $x^v \in \cC$ is fibrant whenever $x$ is fibrant and $v$ is cofibrant.
        \item The hom object $\underline{\cC}(x,y) \in \cV$ is fibrant whenever $x$ is cofibrant and $y$ is fibrant.
    \end{enumerate}
\end{lem}

For the next statement, see \cite{Hovey99}*{Theorems~4.3.2, 4.3.4}, \cite{Riehl14}*{Theorems~10.2.6, 10.2.12}, or \cite{MayP12}*{Remark~16.4.13}.

\begin{prop}\label{pr:Hocot}
    Let $\cV$ be a symmetric monoidal model category.
	\begin{enumerate}
		\item The homotopy category $\Ho(\cV)$ is a closed symmetric monoidal category.
		\item If $\cC$ is a $\cV$-model category, then the homotopy category $\Ho(\cC)$ %
		is enriched, tensored, and cotensored over $\Ho(\cV)$.
	\end{enumerate}
\end{prop}

\begin{defn}\label{def:wmqa}
	For $\cV$ and $\cW$ two monoidal model categories, a \Def{weak (or lax) monoidal Quillen adjunction} is a Quillen adjunction $F \colon \cW \rla \cV \colon G$ %
	equipped with the structure of a lax monoidal functor on the right adjoint $G$ %
	such that the induced oplax monoidal structure %
	on the left adjoint $F$ satisfies the following properties.
		\begin{enumerate}
			\item[(i)] For all cofibrant objects $x,y\in \cW$ the oplax monoidal transformation $F(x\otimes y)\xrightarrow[\sim]{\delta_{x,y}} F(x)\otimes F(y)$ is a weak equivalence in $\cV$.
			\item[(ii)] For some (hence any) cofibrant replacement of the tensor unit $q\colon Q\unit_{\cW} \to \unit_{\cW}$ in $\cW$, the composition $\xymatrix{F(Q\unit_{\cW}) \ar[r]^-{F(q)} & F(\unit_{\cW}) \ar[r]^-{\varepsilon} & \unit_{\cV}}$, with the  oplax monoidal counit $\varepsilon$, is a weak equivalence in $\cV$.
		\end{enumerate}	
	This is called a \Def{strong monoidal Quillen adjunction} if $F$ is a strong monoidal functor, that is the maps $\delta_{x,y}$ and $\varepsilon$ are isomorphisms. In this case the first condition above on $F$ is vacuous, and the second becomes vacuous if the unit object of $\cW$ is cofibrant. If a weak monoidal Quillen adjunction is also a Quillen equivalence it is called a \Def{weak monoidal Quillen equivalence}.
\end{defn}

\begin{rem}\label{rem:StrongMonoidal}
When the left adjoint $F \colon \cW \to \cV$ is strong monoidal, the conditions of a (weak or strong) monoidal Quillen pair reduce to the map $F(q) \colon F(Q\unit_{\cW}) \to F(\unit_{\cW})$ being a weak equivalence.
\end{rem}

For $R$ a commutative ring, the Dold--Kan correspondence gives rise to various weak monoidal Quillen equivalences that are not strong. %

\begin{ex}\label{ex:DK-WeakMonoidal} 
    Viewing the normalization as left adjoint, the adjunction $N \colon s\Mod{R} \rla \Ch_{\geq 0}(R) \colon \Ga$ is a weak monoidal Quillen equivalence with respect to the standard model structure on $s\Mod{R}$ (with weak equivalences and fibrations as in underlying simplicial sets) and the projective model structure on $\Ch_{\geq 0}(R)$ \cite{SchwedeS03equ}*{\S 4.3}. This also holds if we use the Hurewicz model structures on $s\Mod{R}$ and $\Ch_{\geq 0}{R}$, since the Alexander--Whitney map $\AW \colon N(A \ot B) \to N(A) \ot N(B)$ is a chain homotopy equivalence for all $A,B \in s\Mod{R}$. 

    Similarly, viewing the normalization as right adjoint, the adjunction $\Gamma \dashv N$ is a weak monoidal Quillen equivalence. This holds both for the standard projective model structures and the Hurewicz model structures \cite{SchwedeS03equ}*{\S 4.2}.
\end{ex}

\begin{ex}
    There is a weak monoidal Quillen equivalence between categories of symmetric spectra 
	\[		
		\Sp^{\Si}(\Ch_{\geq 0}(R)) \rla \Sp^{\Si}(s\Mod{R})
	\]
	that is a step in the stable Dold--Kan correspondence \cite{Shipley07}*{Proposition~4.4}. The right adjoint is given by (levelwise) normalization followed by a certain restriction functor.
\end{ex}

\begin{ques}
    Does \cite{Shipley07}*{Proposition~4.4} generalize to a weak monoidal Quillen pair between symmetric spectra objects $\Sp^{\Si}(\cW) \rla \Sp^{\Si}(\cV)$ for other weak monoidal Quillen pairs $F \colon \cW \rla \cV \colon G$?
\end{ques}

\begin{ex}
    The category of non-negatively graded cochain complexes $\Ch^{\geq 0}(R)$ admits the injective and Hurewicz model structures \cite{Bousfield03}*{\S 4.4} \cite{Ngopnang24hur}*{\S 6}, as well as the projective model structure \cite{CastiglioniC04}*{Definition~4.7} \cite{Jardine97}. 
    Using the corresponding model structures on cosimplicial modules via the dual Dold--Kan correspondence $N \colon c\Mod{R} \rla \Ch^{\geq 0}(R) \colon \Ga$, both adjunctions $N \dashv \Ga$ and $\Ga \dashv N$ become weak monoidal Quillen equivalences, for a total of six examples. 
        
    In \cite{CastiglioniC04}, Castiglioni and Corti\~nas replace $\Ga$ with a 
    derived-equivalent functor to produce a strong monoidal Quillen equivalence $\Ch^{\geq 0}(R) \rla c\Mod{R}$,
    cf.\ \cite{WhiteY19}*{\S 7.5}. 
\end{ex}

\begin{ex}
In \cite{DwyerK85}, Dwyer and Kan describe the category of \emph{duplicial $R$-modules} $\Mod{R}^{K^{\opp}} = \Fun(K^{\opp},\Mod{R})$, where $K$ is a small category containing both $\De$ and $\De^{\opp}$, so that a duplicial module $A$ has both an underlying simplicial module and a cosimplicial module on the same underlying graded module $\{ A_n \mid n \geq 0 \}$. On the other hand, a \emph{duchain complex} of $R$-modules $C$ has both a chain complex structure $\del \colon C_n \to C_{n-1}$ and cochain complex structure $\delta \colon C_n \to C_{n+1}$ on the same underlying graded module $\{ C_n \mid n \geq 0 \}$. There is a Dold--Kan type equivalence 
\[
N \colon \Mod{R}^{K^{\opp}} \rla R(\del,\delta) \colon \Gamma
\]
between duplicial modules and duchain complexes. %
Both adjunctions $N \dashv \Gamma$ and $\Gamma \dashv N$ are weak monoidal Quillen equivalences \cite{DwyerK85}*{Theorems~3.5, 5.1}.
\end{ex}

\begin{rem}
There is also a cubical Dold--Kan correspondence, between cubical $R$-modules with connections and $\Ch_{\geq 0}(R)$ \cite{BrownH03} \cite{BrownHS11}*{\S 14.8}. It seems plausible that this should be a weak monoidal Quillen equivalence. 
\end{rem}

\begin{ex}
Let $\cV$ be a model category that is a monoidal model category with respect to two monoidal structures $\boxtimes$ and $\ot$. The identity adjunction
\[
F = \Id \colon (\cV,\boxtimes) \rla (\cV,\ot) \colon \Id = G\\
\]
is automatically a Quillen equivalence, and it is weak monoidal if and only if there is a natural map $\mu \colon x \boxtimes y \to x \ot y$ that is a weak equivalence for $x,y \in \cV$ cofibrant, and a weak equivalence $\eta \colon \unit_{\boxtimes} \ral{\sim} \unit_{\ot}$, satisfying associativity and unitality. 
This situation arises for instance when comparing certain Day convolutions with pointwise products in diagram categories \cite{SagaveS13}*{Proposition~2.27} \cite{SagaveS21}*{Remark~3.6} \cite{Schwede18}*{Theorem~1.3.2}.
\end{ex}

\section{Preservation of tensoring and cotensoring}\label{sec:Ptc}

Given an adjunction $F \colon \cW \rla \cV \colon G$, the following lemma gives us a necessary and sufficient condition for which the change of enrichment along the right adjoint $G$ preserves the underlying category, that is, $(G_*\cC)_0 \cong \cC_0$ for any $\cV$-category $\cC$. In this case, any potential model structure on $\cC_0$ is preserved.

\begin{lem}\label{lem:sameunderlying}
	Let $F \colon \cW \rla \cV \colon G$. be an adjunction where the right adjoint $G$ is lax monoidal. The following are equivalent.
	\begin{enumerate}
		\item \label{undercat1} The map $\ep \colon F(\unit_{\cW}) \ral{\cong} \unit_{\cV}$ is an isomorphism. %
		\item \label{undercat2} $G$ commutes with the underlying set functors, that is, the following diagram commutes (up to natural isomorphism)
		\begin{equation}\label{underset}
			\xymatrix{
			\cV \ar[r]^-G \ar[dr]_-{\cV(\unit_{\cV},-)=U} & \cW \ar[d]^-{\cW(\unit_{\cW},-)=U} \\
			& \Set,
			}
		\end{equation}
        where $U$ is the underlying set functor.
		\item \label{undercat3} Change of enrichment along $G\colon\cV\to \cW$ preserves the underlying  category for every $\cV$-enriched category $\cC$, that is, the following diagram commutes (up to natural isomorphism)
		\begin{equation}\label{undercat}
			\xymatrix{
				\CAT{\cV} \ar[r]^-{G_*} \ar[dr]_-{U_*} & \CAT{\cW} \ar[d]^-{U_*} \\
				& \Cat,
			}
		\end{equation}
		\noindent where $U_*$ is the forgetful functor.
	\end{enumerate}
\end{lem}

\begin{proof}
	(1)$\iff$(2) The adjunction $(F\dashv G)$ gives us the following natural bijection
	\begin{equation}\label{under1}
		 \cV(F(\unit_{\cW}), v) \cong \cW(\unit_{\cW}, Gv) \quad \text{for all } v \in \cV,
	\end{equation}
	and the commutativity of \eqref{underset} gives us the following natural bijection
		\begin{equation}\label{under2}
		 \cV(\unit_{\cV}, v) \cong \cW(\unit_{\cW}, Gv) \quad \text{for all } v \in \cV.
	\end{equation}
	Hence $(\ref{under1})$ and \eqref{under2} gives us the following natural bijection
	\begin{equation*}
		 \cV(F(\unit_{\cW}), v)  \cong \cV(\unit_{\cV}, v) \iff F(\unit_{\cW})  \cong \unit_{\cV},\ \text{by the Yoneda lemma}.
	\end{equation*}
 
	\noindent (2)$\implies$(3) Since \Eqref{undercat2} is equivalent to \Eqref{undercat1}, $F(\unit_{\cW})  \cong \unit_{\cV}$ holds. Then the underlying categories are preserved and so the commutativity of $(\ref{undercat})$ holds on objects. On hom objects, we have
	\begin{align*}
		\underline{U_*G_*\cC}(x,y) & := \underline{(UG)_*\cC}(x,y), \ \text{by Lemma~\ref{lem:Compochange}} \\
								   & \cong \underline{U_*\cC}(x,y), \ \text{for all } x,y \in \cC,
	\end{align*}
	\noindent natural since $(\ref{underset})$ is.
	
		\noindent (3)$\implies$(2) The category $\cV$ is $\cV$-enriched as a monoidal category and we have 
		\begin{equation}\label{l4132}
		\xymatrix @R-0.4pc {
		& \underline{U_*G_*\cV}(\unit_{\cV},v) \ar[d]_-{:=} \ar[r]^-{\cong} & \underline{U_*\cV}(\unit_{\cV},v) \ar[d]^-{:=} & \\
		& UG\underline{\cV}(\unit_{\cV}, v) \ar[d]_-{\cong} & U\underline{\cV}(\unit_{\cV}, v) \ar[d]^-{\cong} & \\
		\cW(\unit_\cW,Gv) \ar@{=}[r] & : UGv \ar@{-->}[r]_-{\therefore\ \textcolor{blue}{\cong}} & Uv : \ar@{=}[r] & \cV(\unit_\cV,v),
		}
		\end{equation}
		\noindent for any $v\in \cV$, that is, the solid diagram induces an isomorphism $UGv \cong Uv$ given by the dashed map. All the maps in diagram \eqref{l4132} are natural in $v$, hence so is the isomorphism $UGv \cong Uv$.
\end{proof}

The equivalence (1)$\iff$(2) is also found in \cite{Kelly74}*{Proposition~2.1}. The implication (1)$\implies$(3) is mentioned in \cite{GuillouM20}*{Proposition~3.2}.

\begin{cor}\label{cor:StrongUnderlying}
    For any adjunction $F\dashv G$ where the left adjoint $F$ is strong monoidal, the right adjoint $G$ preserves underlying sets.
\end{cor}

\begin{ex} 
The normalization $N$ and the denormalization $\Gamma$ defining the \mbox{Dold--Kan} correspondence $N \colon s\Mod{R} \rla \Ch_{\geq 0}(R) \colon \Ga$ %
are both not strong monoidal, but both preserve the tensor unit $\unit$, hence both preserve underlying sets by Lemma~\ref{lem:sameunderlying}. We can also see this directly. 
The underlying set of a simplicial $R$-module $A$ is given by its \mbox{$0$-simplices}, i.e., $U(A) = A_0$. 
The underlying set of a non-negatively graded chain complex $C$ is given by its $0$-cycles, i.e., $U(C) = Z_0(C) = C_0$. 
In particular, we have: 
\[
\begin{cases}
U(N(A)) = Z_0(N(A)) = A_0 = U(A) &\text{for all } A \in s\Mod{R} \\
U(\Gamma(C)) = \Gamma(C)_0 = C_0 = U(C) &\text{for all } C \in \Ch_{\geq 0}(R). \\
\end{cases}
\]
\end{ex}

\begin{rem}
Even when a left adjoint $F$ is strong monoidal, $F$ itself need not preserve underlying sets. For example, 
the geometric realization $|\cdot| \colon s\Set \to \Top$ is strong monoidal but does not preserve underlying sets, i.e., in general $U(|X_\bullet|) \ncong U(X_\bullet) = X_0$. 
For instance, the standard \mbox{$1$-simplex} $\De^{1}$ has underlying set $(\De^1)_0 = \{0,1\}$, whereas its geometric realization $\abs{\De^1}$ has underlying set an uncountable interval. 
Note that geometric realization $|\cdot|$ does not admit a left adjoint, since it doesn't preserve infinite products.
\end{rem}

The following proposition gives us a necessary and sufficient condition (namely, $F$ being strong monoidal), under which the tensoring or the cotensoring is preserved by a change of enrichment along the right adjoint of the lax monoidal adjunction $F \colon \cW \rla \cV \colon G$.

\begin{prop}\label{pr:coten-strong}
    	Let $\cV$ and $\cW$ be closed symmetric monoidal categories and $F \colon \cW \rla \cV \colon G$ an adjunction such that the right adjoint $G$ is lax monoidal.
    \begin{enumerate}
        \item \label{item:StrongTensoring} If $F$ is strong monoidal, then for a $\cV$-category $\cC$, the $\cW$-category $G_*\cC$ admits a tensoring or a cotensoring over $\cW$ given respectively by:
        \begin{equation*}
        x\otimes w := x\otimes Fw \quad \text{and} \quad x^w := x^{Fw}, \quad \text{for any } x \in \cC \text{ and } w \in \cW.
         \end{equation*}
        \item \label{item:NeedStrong} If $\ep \colon F(\unit_{\cW}) \ral{\cong} \unit_{\cV}$ is an isomorphism and the $\cW$-category $G_*\cV$ admits a tensoring or a cotensoring over $\cW$, then $F$ is strong monoidal.
    \end{enumerate}
\end{prop}

\begin{proof}
    \Eqref{item:StrongTensoring} See \cite{Riehl14}*{Proposition~3.7.11}.
    
	\Eqref{item:NeedStrong} Assume that $F(\unit_{\cW})\cong \unit_{\cV}$ and the $\cW$-category $G_*\cV$ admits a tensoring over $\cW$. We have the  following bijection for any $w \in \cW$ and $v_1, v_2 \in \cV$.
	\begin{align*}
		\Hom_{G_*\cV}(v_1\otimes w,v_2) & \cong \Hom_{\cW}(w,\underline{G_*\cV}(v_1,v_2)),\ \text{using the tensoring over $\cW$} \\
        \Hom_{\cV}(v_1 \otimes w,v_2) & \cong \Hom_{\cW}(w,G\underline{\cV}(v_1,v_2)),\ (G_*\cV)_0\cong\cV_0\ \text{since}\ F(\unit_{\cW})\cong \unit_{\cV} \\ %
		& \cong \Hom_{\cV}(Fw,\underline{\cV}(v_1,v_2)),\ \text{by the unenriched adjunction}\ F\dashv G \\
		& \cong \Hom_{\cV}(v_1\otimes Fw,v_2),\ \text{by the unenriched tensor-hom adjunction} \\
		\implies v_1 \otimes w & \cong v_1\otimes Fw,\ \text{by the unenriched Yoneda lemma applied to}\ \cV. %
	\end{align*}
	In particular for $v_1 = \unit_{\cV}, \unit_{\cV} \otimes w \cong \unit_{\cV} \otimes F(w) \cong F(w)$. Applying the associativity of the tensoring \cite{Riehl14}*{Lemma~3.7.7}, 
	we have the following for any $w_1,w_2 \in \cW$:
	\begin{equation*}
		\xymatrix @R-0.4pc {
		(\unit_{\cV} \otimes w_1) \otimes w_2 \ar[r]^-\cong \ar[d]_-\cong & \unit_{\cV} \otimes (w_1 \otimes w_2) \ar[d]^-\cong \\
			F(w_1) \otimes w_2 \ar[d]_-\cong & F(w_1\otimes w_2) \ar@{=>}[r] & F(w_1\otimes w_2) \cong F(w_1)\otimes F(w_2).\\
			F(w_1)\otimes F(w_2) \ar@{-->}[ur]_-{\textcolor{blue}{\therefore\ \cong}} & \\
		}
	\end{equation*}
	The isomorphism $F(w_1\otimes w_2) \cong F(w_1) \otimes F(w_2)$ is natural, associative, and unital. 

	The proof in the case of a cotensoring is similar.
\end{proof}

\begin{ex}\label{ex:DGcat}
Let $R$ be a commutative ring and $\cC$ a DG-category over $R$, i.e., a category enriched over $\Ch(R)$. By Lemma~\ref{lem:GoodTruncation}, the good truncation $\tau_{\geq 0} \colon \Ch(R) \to \Ch_{\geq 0}(R)$ has a strong monoidal left adjoint $\io$. By Proposition~\ref{pr:coten-strong}~\Eqref{item:StrongTensoring}, truncating each hom complex in $\cC$ yields a category $(\tau_{\geq 0})_* \cC$ that is enriched, tensored, and cotensored over $\Ch_{\geq 0}(R)$, with the same underlying category as $\cC$ (by Corollary~\ref{cor:StrongUnderlying}).

However, the $s\Mod{R}$-enriched category $\Ga_*(\tau_{\geq 0})_* \cC$ admits neither a tensoring nor a cotensoring over $s\Mod{R}$, by Proposition~\ref{pr:coten-strong}~\Eqref{item:NeedStrong}. Indeed, the left adjoint $N$ is oplax monoidal but not strong, i.e., the Alexander--Whitney map $\AW \colon N(A \ot B) \to N(A) \ot N(B)$ is not an isomorphism. It was observed in \cite{Lurie17}*{Warning~1.3.5.4} that the formula 
\[
X \ot A := X \ot N(A) \quad \text{for all } X \in \cC, \ A \in s\Mod{R}
\]
does not define a tensoring of $\Ga_*(\tau_{\geq 0})_* \cC$ over $s\Mod{R}$ for that reason.
\end{ex}

\section{Weak enriched adjunction}\label{sec:wea}

The proposition below gives a necessary and sufficient condition under which the lax monoidal adjunction $F \colon \cW \rla \cV \colon G$ lifts to a \mbox{$\cW$-adjunction} $F \colon \cW \rla G_* \cV \colon G$.

\begin{prop}\label{pr:prob8}
	The lax monoidal adjunction $F \colon \cW \rla \cV \colon G$ lifts to a \mbox{$\cW$-adjunction} $F \colon \cW \rla G_* \cV \colon G$ if and only if the left adjoint $F$ is strong monoidal.
\end{prop}

\begin{proof}
($\implies$) Let us show that the functor $F$ satisfies the natural isomorphism $F(w \otimes w') \cong Fw \otimes Fw'$, for any $w, w' \in \cW$.
 \begin{align*}
     \underline{G_*\cV}(F(w\otimes w'),v)       
& \cong \underline{\cW}(w\otimes w',Gv)\ \text{by enriched }F\dashv G \\
& \cong \underline{\cW}(w, \underline{\cW}(w',Gv))\ \text{by the tensor-hom adjunction on }\cW \\
& \cong \underline{\cW}(w,\underline{G_*\cV}(Fw',v))\ \text{by enriched }F\dashv G \\
& := \underline{\cW}(w,G\underline{\cV}(Fw',v))\ \text{by definition} \\
& \cong \underline{G_*\cV}(Fw, \underline{\cV}(Fw',v))\ \text{by  enriched }F\dashv G \\
& \cong \underline{G_*\cV}(Fw\otimes Fw',v)\ \text{by the tensor-hom adjunction on }\cV.
 \end{align*}
 Therefore, $F(w\otimes w')\cong Fw\otimes Fw'$ by the enriched Yoneda lemma and the naturality is given by the naturality of the unenriched adjunction $F\dashv G$. \\
($\impliedby$) The converse is found in \mbox{\cite{Riehl14}*{Corollary~3.7.12}}. 
\end{proof}

In the case of a weak monoidal Quillen adjunction $F \colon \cW \rla \cV \colon G$, the left adjoint $F$ is not necessarily strong. The adjunction $F \dashv G$ will lift to a weak version of the $\cW$-adjunction $F \colon \cW \rla G_*\cV \colon G$ defined as follows.

\begin{defn}
    Let $\cC$ and $\cD$ be $\cV$-model categories. A \Def{weak $\cV$-adjunction} $F \colon \cC \rla \cD \colon G$ consists of a natural map $\underline{\cD}(Fc,d) \to \underline{\cC}(c,Gd)$ in $\cV$ that is a weak equivalence for any cofibrant $c \in \cC$ and fibrant $d \in \cD$.
\end{defn}

\begin{prop}\label{pr:prob9}\label{pr:WeakAdjunction}
	Let $\cV$ and $\cW$ be symmetric monoidal model categories. Any weak monoidal Quillen adjunction $F \colon \cW \rla \cV \colon G$ lifts to a weak $\cW$-adjunction $F \colon \cW \rla G_*\cV \colon G$.
\end{prop}

\begin{proof}
    For any $w_1 \in \cW$, any cofibrant $w_2 \in \cW$ and any fibrant $v \in \cV$, we have  the following.
    \begin{align}
    \cW(w_1,\underline{G_*\cV}(Fw_2,v)) & := \cW(w_1,G\underline{\cV}(Fw_2,v)),\ \text{by definition} \nonumber \\
                           & \cong \cV(Fw_1,\underline{\cV}(Fw_2,v)),\ \text{by the adjunction}\ F\dashv G \nonumber \\
                           & \cong \cV(Fw_1\otimes Fw_2, v),\ \text{by the tensor-hom adjunction} \nonumber \\
                           & \xrightarrow[]{\delta^*} \cV(F(w_1\otimes w_2), v),\ \text{$\delta$ the comultiplication map of $F$} \nonumber \\
                           & \cong \cW(w_1\otimes w_2, Gv),\ \text{by the adjunction}\ F\dashv G \nonumber \\
                           & \cong \cW(w_1, \underline{\cW}(w_2, Gv)),\ \text{by the tensor-hom adjunction}. \label{eq:IsoWeakAdj}
  \end{align}
  For $w_1 = \underline{G_*\cV}(Fw_2,v)$, the image of the identity gives us the comparison map
  \begin{equation*}
      \underline{G_*\cV}(Fw_2,v)\xrightarrow[]{\phi_{w_2,v}} \underline{\cW}(w_2, Gv).
  \end{equation*}
  For $w_1,w_2$ cofibrant in $\cW$, $F(w_1\otimes w_2)\xrightarrow[\sim]{\delta}Fw_1\otimes Fw_2$ is a weak equivalence in $\cV$. Then applying $\underline{\cV}(-,v)$, with fibrant $v \in \cV$, we obtain the following weak equivalence
  \begin{equation*}
     \underline{\cV}(Fw_1\otimes Fw_2,v) \xrightarrow[\sim]{\delta^*} \underline{\cV}(F(w_1\otimes w_2),v),
  \end{equation*}
  since $\underline{\cV}(-,v)$, with $v$ fibrant, preserves weak equivalences between cofibrant objects in $\cV$ \cite{Riehl14}*{Lemma~9.2.3}.

  We want to show that the map $\phi_{w_2,v}$ constructed above is a weak equivalence when $w_2 \in \cW$ is cofibrant and $v\in \cV$ is fibrant. By saturation (Proposition~\ref{pr:Saturation}), %
  it suffices to show that the induced map 
  $[w_1,\underline{G_*\cV}(Fw_2,v)] \xrightarrow[]{(\phi_{w_2,v})_*} [w_1,\underline{\cW}(w_2, Gv)]$
  of hom-sets in $\Ho(\cW)$ is a bijection for any cofibrant $w_1 \in \cW$. By Proposition~\ref{pr:Hocot}, $\Ho(\cV)$ and $\Ho(\cW)$ are closed symmetric monoidal categories, which gives the following:
\begin{align*}
    [w_1,\underline{G_*\cV}(Fw_2,v)] & := [w_1,G\underline{\cV}(Fw_2,v)],\ \text{by definition} \\
                                         & \cong [w_1,\mathbf{R}G\left(\underline{\cV}(Fw_2,v)\right)],\ \text{since $\underline{\cV}(Fw_2,v)$ is fibrant} \\
                                         & \cong [\mathbf{L}F(w_1),\underline{\cV}(Fw_2,v)],\ \text{by the derived adjunction} \\
                                         & \cong [\mathbf{L}F(w_1)\otimes^{\mathbf{L}} F(w_2),v],\ \text{by the derived tensor-hom in}\ \Ho(\cV) \\
                                         & \cong [F(w_1)\otimes F(w_2),v],\ \text{since $w_1$ and $w_2$ are cofibrant} \\
                                         & \xrightarrow[\cong]{\delta^*} [F(w_1\otimes w_2),v],\ \text{since $\delta$ is a weak equivalence} \\
                                         & \cong [\mathbf{L}F(w_1\otimes w_2),v],\ \text{since $w_1\otimes w_2$ is cofibrant} \\
                                         & \cong [w_1\otimes w_2,\mathbf{R}G(v)],\ \text{by the derived adjunction} \\
                                         & \cong [w_1, \underline{\cW}(w_2, \mathbf{R}G(v))],\ \text{by the derived tensor-hom in}\ \Ho(\cW) \\
                                         & \cong [w_1, \underline{\cW}(w_2, Gv)],\ \text{since $v$ is fibrant},
\end{align*}
where $\mathbf{R}G$ and $\mathbf{L}F$ are respectively the right and the left derived functor of the adjunction \mbox{$F \dashv G$}.
\end{proof}

The following lemma is used in the proof of Proposition~\ref{pr:PropW_ten_coPreserv}.

\begin{lem}\label{lem:lastLem} 
Consider the map $\cW(w_1,\underline{G_*\cV}(Fw_2,v))\to \cW(w_1, \underline{\cW}(w_2,Gv))$ constructed in Equation~\eqref{eq:IsoWeakAdj}. In the special case $w_1 = Gv$ and $w_2 = \unit_{\cW}$, we have
\begin{align*}
    \cW(Gv,\underline{G_*\cV}(F\unit_{\cW},v)) & \to \cW(Gv, \underline{\cW}(\unit_{\cW},Gv)) \\
    G(\varepsilon^*\circ \eta) & \mapsto \Tilde{\eta},
\end{align*}
with $\xymatrix{v \ar[r]^-{\eta}_-\cong & \underline{\cV}(\unit_{\cV},v) \ar[r]^-{\varepsilon^*} & \underline{\cV}(F\unit_\cW, v)}$ and where $\tilde{\eta} \colon Gv \to \ul{\cW}(\unit_{\cW}, Gv)$ denotes the canonical isomorphism in $\cW$.
\end{lem}

\begin{proof}
    From the series of maps in \eqref{eq:IsoWeakAdj}, diagram chasing gives the following:
    \begin{equation*}
        \xymatrix @R-0.7pc {
             & & \cW(Gv,G\underline{\cV}(F(\unit_{\cW}),v)) & & \ni & G(\varepsilon^*\circ \eta) \ar@{|->}[d] \\
             & \ar[r]^-{\cong} &  \cV(FGv,\underline{\cV}(F(\unit_{\cW}),v)) &  & \ni & \varepsilon^*\circ \eta \circ \epsilon_{v} \ar@{|->}[d] \\
                                              & \ar[r]^-{\cong} &  \cV(FGv\otimes F(\unit_{\cW}), v) &  & \ni & \mathrm{ev}\circ (\eta\otimes \mathrm{id}) \circ (\epsilon\otimes \varepsilon) \ar@{|->}[d] \\
                                              & \ar[r]^-{\delta^*} & \cV(F(Gv\otimes \unit_{\cW}), v) \ar[d]_-{\cong} & \ar[l]_-{\cong}^-{F(\Tilde{\rho})^*} \cV(FGv, v) \ar[d]_-{\cong} & \ni &  \epsilon \ar@{|->}[d] \\
                                              & \ar[r]^-{\cong} & \cW(Gv\otimes \unit_{\cW}, Gv) & \ar[l]_-{\cong}^-{\Tilde{\rho}^*} \cW(Gv, Gv) & \ni &  \mathrm{id}\ar@{|->}[d] \\
                                              & \ar[r]^-{\cong} & \cW(Gv, \underline{\cW}(\unit_{\cW}, Gv)) &  & \ni & \Tilde{\eta}. 
        }
    \end{equation*}
    The step with $\delta^*$ relied on the counitality equation of the oplax monoidal functor $F$.
\end{proof}

The change of enrichment along a weak monoidal Quillen adjunction preserves the SM7 axiom, whenever the left adjoint preserves the tensor unit as stated by the following lemma.

\begin{lem}\label{lem:G*SM7}
    Let $\cV$ and $\cW$ be symmetric monoidal model categories. For any weak monoidal Quillen adjunction $F \colon \cW \rla \cV \colon G$ such that $\ep \colon F(\unit_{\cW}) \ral{\cong} \unit_{\cV}$ is an isomorphism, the change of enrichment along $G$ preserves SM7.
\end{lem}

\begin{proof}
    Consider a $\cV$-enriched model category $\cC$ satisfying SM7. The category $G_*\cC$ is a $\cW$-enriched category since $G$ is a lax monoidal functor. Also, $G_*\cC$ has the same underlying category as $\cC$ by Lemma~\ref{lem:sameunderlying}, i.e., $(G_*\cC)_0 = \cC_0$, since \mbox{$F(\unit_{\cW}) \cong \unit_{\cV}$}. In particular, the underlying category $(G_*\cC)_0$ of $G_*C$ has the same model structure as $\cC_0$. Let us show that the model category $G_*\cC$ satisfies SM7. Let $i \colon a \hookrightarrow b$ be a cofibration and $p \colon x \twoheadrightarrow y$ be a fibration in $(G_*\cC)_0 = \cC_0$. We have
	\begin{equation*}
		\xymatrix @R-0.4pc {
			\underline{G_*\cC}(b,x) \ar[r]^-{(i^*,p_*)} \ar@{=}[dd]_-{\text{\textcircled1}} & \underline{G_*\cC}(a,x)\times_{\underline{G_*\cC}(a,y)} \underline{G_*\cC}(b,y) \ar@{=}[d]^-{\text{\textcircled3}} \\
		& G\underline{\cC}(a,x)\times_{G\underline{\cC}(a,y)} G\underline{\cC}(b,y) \\
			G\underline{\cC}(b,x) \ar[r]_-{G(i^*,p_*)} & G \left( \underline{\cC}(a,x)\times_{\underline{\cC}(a,y)} \underline{\cC}(b,y)\right) \ar[u]^-{\cong}_-{\text{\textcircled2}}
		}
	\end{equation*}
where the \textcircled{1} and \textcircled{3} are given by the definition of the change of enrichment. The isomorphism \textcircled{2} holds since $G$ preserves limits as a right adjoint functor. Since the $\cV$-enriched model category $\cC$ satisfies SM7, the map 
 \[(i^*,p_*) \colon \ul{\cC}(b,x) \to \ul{\cC}(a,x) \times_{\ul{\cC}(a,y)} \ul{\cC}(b,y)\] 
 is a fibration. Hence $G(i^*,p_*)$ is also a fibration since $G \colon \cV \to \cW$ is a right Quillen functor. The same argument works for acyclic fibrations.
\end{proof}

\section{Weak tensoring and weak cotensoring}\label{sec:swtwc}

\begin{defn}\label{def:wtwc}
	Let $\cC$ be a $\cV$-enriched model category satisfying SM7. 
 \begin{enumerate}
    \item A \Def{weak tensoring} of $\cC$  over $\cV$ is a bifunctor $- \otimes - \colon \cC \times \cV \to \cC$, satisfying the external pushout-product axiom, together with a natural map 
     \begin{equation}
         \varphi_{v,x,y} \colon \underline{\cC}(x \otimes v,y) \to \underline{\cV}(v,\underline{\cC}(x,y))
     \end{equation}
     which is a weak equivalence in $\cV$ for cofibrant $v \in \cV$, cofibrant $x \in \cC$, and fibrant $y \in \cC$, and a natural map $\rho_x \colon x \otimes \unit \to x$, for any $x \in \cC$, making the following commute:
     \begin{equation*}
         \xymatrix{
         \underline{\cC}(x\otimes \unit, y) \ar[d]_-{\varphi_{\unit,x,y}} & \ar[l]_-{\rho_x^*} \underline{\cC}(x,y) \ar[dl]^-\cong \\ 
         \underline{\cV}(\unit,\underline{\cC}(x,y)).}
     \end{equation*}
     \item A \Def{weak cotensoring} of $\cC$  over $\cV$ is a bifunctor $(-)^-\colon\cC\times \cV^{\mathrm{op}} \to \cC$, satisfying the external pullback-power axiom, together with natural map 
     \begin{equation}
         \psi_{v,x,y}\colon\underline{\cC}(x,y^v)\to \underline{\cV}(v,\underline{\cC}(x,y))
     \end{equation}
     which is a weak equivalence in $\cV$ for cofibrant $v \in \cV$, cofibrant $x \in \cC$ and fibrant $y \in \cC$, and a natural map $\eta_y \colon y \to y^\unit$, for any $y \in \cC$, making the following diagram commute:
     \begin{equation*}
         \xymatrix{
         \underline{\cC}(x, y^\unit) \ar[d]_-{\psi_{\unit,x,y}} & \ar[l]_-{(\eta_y)_*} \underline{\cC}(x,y) \ar[dl]^-\cong \\
         \underline{\cV}(\unit,\underline{\cC}(x,y)).
         }
     \end{equation*}
    \item A weak tensoring is said to be:
 \begin{itemize}
    \item \Def{set-compatible} if the structure map $\varphi_{v,x,y}$ induces a bijection of underlying sets
    \[
    U\varphi_{v,x,y} \colon U\underline{\cC}(x \otimes v,y) \ral{\cong} U\underline{\cV}(v,\underline{\cC}(x,y))
    \]
    for all $v \in \cV$ and $x,y \in \cC$.
    \item \Def{unital} if the natural map $\rho_x \colon x \otimes \unit \to x$ is an isomorphism for all $x \in \cC$.
\end{itemize}
Similarly for a weak cotensoring.
 \end{enumerate}
\end{defn}

\begin{ex}\label{ex:ct2wct}
    Tensoring and cotensoring are a special kind of weak tensoring and weak cotensoring, respectively, 
    which are moreover set-compatible and unital.
\end{ex}

\begin{rem}
    Given a weak tensoring that is set-compatible, the natural bijection $\cC(x \otimes v,y) \cong \cV(v,\underline{\cC}(x,y))$ yields an unenriched adjunction $x \otimes - \colon \cV \rla \cC \colon \underline{\cC}(x,-)$. However, the left adjoint $x \otimes -$ need not lift to a $\cV$-functor, cf.\ Remark~\ref{rem:UnenrichedTensorHom}.
\end{rem}

The change of enrichment along a weak monoidal Quillen adjunction preserves the weak tensoring and the weak cotensoring, whenever the left adjoint preserves the tensor unit as stated by the following result.

\begin{prop}\label{pr:PropW_ten_coPreserv}
	Let $\cV$ and $\cW$ be symmetric monoidal model categories. For any weak monoidal Quillen adjunction $F \colon \cW \rla \cV \colon G$ such that $\ep \colon F(\unit_{\cW}) \ral{\cong} \unit_{\cV}$ is an isomorphism, the change of enrichment $G_*$ along the right adjoint $G \colon \cV \to \cW$ preserves the weak tensoring and the weak cotensoring. 
    
    If moreover the weak tensoring on $\cC$ is set-compatible (resp.\ unital), then so is the induced weak tensoring on $G_*\cC$; likewise for the weak cotensoring.
\end{prop}

\begin{proof}
\textbf{SM7:} Note that $G_*\cC$ satisfies SM7, by Lemma~\ref{lem:G*SM7}.

\textbf{Weak tensoring:} Assume that $\cC$ is weakly tensored over $\cV$. Let us show that the bifunctor $- \otimes - \colon G_*\cC \times  \cW \to G_*\cC$ given by $x \otimes w := x \otimes Fw$ satisfies the external pushout-product axiom. Let $i \colon a \to b$ be a (acyclic) cofibration in $\cW$ and $k \colon x \to y$ be a (acyclic) cofibration in $(G_*\cC)_0 = \cC_0$. The pushout-product $i \square k$ of $i$ and $k$ given by
 \begin{equation*}
     \xymatrix{
     x \otimes b \amalg_{x\otimes a} y \otimes a \ar@{=}[d] \ar[r]^-{k \square i} & y \otimes b \ar@{=}[d] \\
     x \otimes Fb \amalg_{x\otimes Fa} y \otimes Fa \ar[r]^-{k \square Fi} & y \otimes Fb \\
     } 
 \end{equation*}
is a (acyclic) cofibration in $(G_*\cC)_0 = \cC_0$ since $Fi \colon Fa \to Fb$ is a (acyclic) cofibration in $\cV$ as $F$ is a left Quillen functor and the bifunctor $- \otimes - \colon \cC \times \cV \to \cC$ satisfies the external pushout-product axiom.
 
 Define the natural map $\Tilde{\varphi}_{w,x,y} \colon \underline{G_*\cC}(x \otimes w,y) \to \underline{\cW}(w,G\underline{\cC}(x,y))$ to be the composite
\begin{equation}\label{eq:InducedTensoring}
     \underline{G_*\cC}(x\otimes w,y) := \underline{G_*\cC}(x\otimes Fw,y)\xrightarrow[]{G\varphi_{Fw,x,y}}\underline{G_*\cV}(Fw,\underline{\cC}(x,y))\xrightarrow[]{}\underline{\cW}(w,G\underline{\cC}(x,y)),
\end{equation}
 where the last arrow is the natural weak equivalence thanks to the existing weak \mbox{$\cW$-adjunction} by Proposition~\ref{pr:prob9}, given by the weak monoidal Quillen adjunction $F \dashv G$. For a cofibrant object $w \in \cW$, $Fw \in \cV$ is also cofibrant and so the map $\Tilde{\varphi}_{w,x,y}$ is a weak equivalence, for cofibrant objects $x \in \cC$ and $w \in \cW$, and the fibrant object $y \in \cC$, since the map $G\varphi_{Fw,x,y}$ is a weak equivalence as the right Quillen functor $G$ preserves weak equivalences between fibrant objects. Consider the natural map \mbox{$\Tilde{\rho}_x \colon x \otimes \unit_{\cW} \to x$} given by the composite 
 \begin{equation}\label{eq:InducedUnitor}
     \xymatrix{x\otimes \unit_{\cW}: \ar@/_2pc/[rrr]_-{\Tilde{\rho}_x} \ar@{=}[r] & x\otimes F(\unit_{\cW})\ar[r]^-{x\otimes \varepsilon} & x\otimes \unit_{\cV}\ar[r]^-{\rho_x} & x},
 \end{equation}
 where $F\unit_{\cW}\xrightarrow[]{\varepsilon} \unit_{\cV}$ is the counit map of the oplax monoidal functor $F$. The unitality condition follows from the following commutative diagram 
 \begin{equation*}
         \xymatrix @C-0.9pc @R-1.2pc {
            G\underline{\cC}(x \otimes \unit_{\cW},y): \ar@{=}[r] \ar@/_3.5ex/[ddddr]_-{\Tilde{\varphi}} &  G\underline{\cC}(x\otimes F(\unit_{\cW}),y)\ar[dd]_-{G\varphi} & & \ar[ll]_-{G(\mathrm{id}\otimes \varepsilon)^*} G\underline{\cC}(x\otimes \unit_{\cV},y) \ar[dd]^-{G\varphi}  & &  \ar[ll]_-{G\rho^*} G\underline{\cC}(x,y) \ar[ddll]^-{G\eta}_-\cong \\
            & & \text{\textcircled1} & & \\
            & G\underline{\cV}(F(\unit_{\cW}),\underline{\cC}(x,y)) \ar[dd]_-{\psi}  & &  \ar[ll]_-{G\varepsilon^*} G\underline{\cV}(\unit_{\cV},\cC(x,y)) & \\
            & & \text{\textcircled2} & & \\
            & \underline{\cW}(\unit_{\cW},G\underline{\cC}(x,y)) & & \ar[ll]_-{\Tilde{\eta}}^-\cong G\underline{\cC}(x,y), \ar[uu]^-{\cong}_-{G\eta} & 
         }
     \end{equation*}
 where the square \textcircled{1} is commutative by the naturality of $\varphi$ and the square \textcircled{2} is commutative by Lemma~\ref{lem:lastLem}, where $v = \underline{\cC}(x,y)$.
 
\textbf{Weak cotensoring:} Now, assume that $\cC$ is weakly cotensored over $\cV$. 
 The induced weak cotensoring on $G_*\cC$ is the bifunctor $(-)^- \colon G_*\cC \times \cW^{\mathrm{op}} \to G_*\cC$ given by $y^w := y^{Fw}$. Take as natural map \mbox{$\Tilde{\psi}_{w,x,y} \colon \underline{G_*\cC}(x, y^w) \to \underline{\cW}(w,G\underline{\cC}(x,y))$} the composite
 \begin{equation*}
     \underline{G_*\cC}(x, y^w):=G\underline{\cC}(x, y^{Fw})\xrightarrow[]{G\psi_{Fw,x,y}}G\underline{\cV}(Fw,\underline{\cC}(x,y))\xrightarrow[]{}\underline{\cW}(w,G\underline{\cC}(x,y)).
 \end{equation*}
Take as natural map $\Tilde{\eta}_x \colon x \to x^{\unit_{\cW}}$ the composite
 \begin{equation*}
     \xymatrix{
     x \ar@/_1.8pc/[rrr]_-{\Tilde{\eta}_x} \ar[r]^-{\eta_x} & x^{\unit_{\cV}}\ar[r]^-{x^{\varepsilon}} & x^{F(\unit_{\cW})}\ar@{=}[r] & :x^{\unit_{\cW}}. \\
     }
 \end{equation*} 
The verification of the axioms is dual to the weak tensoring case proved above.

\textbf{Set-compatible:} Assume that the structure map $\varphi_{v,x,y} \colon \underline{\cC}(x \otimes v,y) \to \underline{\cV}(v,\underline{\cC}(x,y))$ induces a bijection on underlying sets. Then so does the map $G\varphi_{Fw,x,y}$ in \eqref{eq:InducedTensoring}, since $G$ preserves underlying sets. The next step in the composite $\underline{G_*\cV}(Fw,\underline{\cC}(x,y)) \to \underline{\cW}(w,G\underline{\cC}(x,y))$ induces on underlying sets the bijection of hom-sets $\cV(Fw,\underline{\cC}(x,y)) \cong \cW(w,G\underline{\cC}(x,y))$, by Proposition~\ref{pr:WeakAdjunction}. 

\textbf{Unital:} Assume that the structure map $\rho_x \colon x \ot \unit_{\cV} \ral{\cong} x$ is an isomorphism. Then so is the composite $\tilde{\rho}_x$ in \eqref{eq:InducedUnitor}, since $\ep \colon F(\unit_{\cW}) \ral{\cong} \unit_{\cV}$ was assumed an isomorphism.
\end{proof}

\section{Unit axiom}\label{sec:ua}

In this section, we present some alternate forms of the external unit axiom (see Definition~\ref{def:def-emc}) not necessarily involving the tensoring. %
The following statement 
generalizes a known fact about simplicial model categories \cite{GoerssJ09}*{Proposition~II.3.10} and %
$\cV$-model categories \cite{Riehl14}*{Theorem~10.2.12}.

\begin{prop}\label{pr:GeneGJ09PropII.3.10}
    Let $\cC$ be a $\cV$-enriched model category satisfying SM7 equipped with a weak tensoring and cotensoring that are set-compatible. Let $x$ and $y$ be respectively cofibrant and fibrant objects of $\cC$, and $v$ be a cofibrant object of $\cV$. Then there are natural bijections of hom-sets in $\Ho(\cV)$ given by
    \begin{equation*}
      [v,\underline{\cC}(x,y)] \cong [x\otimes v, y] \quad \text{and} \quad [v,\underline{\cC}(x,y)] \cong [x, y^v].
  \end{equation*}
\end{prop}

\begin{proof}
    We prove the first bijection; the second is dual. 
    Since the weak tensoring is set-compatible, we have a bijection of hom sets:
    \begin{equation*}
    \xymatrix @C+1.5pc {
    \cC(x\otimes v, y) \ar[d]_-{\text{quotient}} \ar[r]^-\cong & \cV(v,\underline{\cC}(x,y)) \ar[d]^-{\text{quotient}} \\
    \cC(x\otimes v, y)/\!\sim \ar[d]_-\cong \ar@{-->}[r]^-{\textcolor{red}{?}}_-{\cong}
 & \cV(v, \underline{\cC}(x,y))/\!\sim \ar[d]^-\cong \\
    [x\otimes v, y] & [v, \underline{\cC}(x,y)],
    }
  \end{equation*}
  where the two downward maps under the quotient maps are bijections since $x \otimes v$ is cofibrant and $\underline{\cC}(x,y)$ is fibrant by Lemma~\ref{lem:cofibcoten}. Hence, it suffices to show that for any $f,g \colon x \otimes v \to y$ in $\cC$ and their corresponding maps $f',g' \colon v \to \underline{\cC}(x,y)$ in $\cV$:
  \begin{equation*}
      f \sim g \ \text{in} \ \cC \iff f' \sim g'\ \text{in} \ \cV.
  \end{equation*}
  ($\implies$) Assume $f \sim g$ in $\cC$. 
There is a right homotopy 
	\begin{equation*}
		\xymatrix @C+0.8pc {
		& y \\
		x \otimes v \ar[r]^-{H} \ar@/^/[ur]^-{f} \ar@/_/[dr]_-{g} & \Path(y) \ar[u]_-{ev_0} \ar[d]^-{ev_1} & \\
		& y \\
		}  	
	\end{equation*}
for some path object for $y$
  \begin{equation*}
      \xymatrix@C+1.2pc{
      y \ar[r]^-{\sim}_-{c} \ar@/_2pc/[rr]_-{\Delta} & \Path(y) \ar@{->>}[r]^-{(ev_0,ev_1)} & y \times y.
      }
  \end{equation*}
  By the (unenriched) tensor-hom adjunction, the right homotopy corresponds to a diagram in $\cV$
  \begin{equation*}
      \xymatrix @C+0.8pc {
      & \underline{\cC}(x,y) \\
      v \ar[r]^-{H'} \ar@/^1pc/[ur]^-{f'} \ar@/_1pc/[dr]_-{g'} & \underline{\cC}(x,\Path(y)) \ar[u]_-{(ev_0)_*} \ar[d]^-{(ev_1)_*} & \\
       & \underline{\cC}(x,y). \\ 
      }
  \end{equation*}
  It suffices to show that 
  \begin{equation*}
  \xymatrix@C+1.5pc{
    \underline{\cC}(x,y) \ar[r]^-{\underline{\cC}(x,c)}_-{=c_*} \ar@/^2.5pc/[rr]^-{\underline{\cC}(x,\Delta)} \ar@/_/[drr]_-\Delta & \underline{\cC}(x,\Path(y)) \ar[r]^-{\underline{\cC}(x,(ev_0,ev_1))}_-{=(ev_0,ev_1)_*} & \underline{\cC}(x,y\times y) \ar[d]^-\cong \\
    & & \underline{\cC}(x,y)\times \underline{\cC}(x,y)
    }
 \end{equation*}
 is a path object for $\underline{\cC}(x,y)$. Since $x$ is cofibrant, then by \cite{Riehl14}*{Lemma~9.2.3} the functor $\underline{\cC}(x,-)\colon \cC\to \cV$ preserves weak equivalences between fibrant objects and any fibration, as a right Quillen adjoint of $x\otimes - $. But $\Path(y)$ is fibrant, since $y\times y$ is, and $(ev_0,ev_1)$ is a fibration. Hence $c_*$ is a weak equivalence and $(ev_0,ev_1)_*$ is a fibration.

 ($\impliedby$) Assume $f'\sim g'$ in $\cV$. Take a left homotopy $H \colon \Cyl(v) \to \underline{\cC}(x,y)$ from $f'$ to $g'$ and proceed with the dual argument as above.
\end{proof}

\begin{lem}
    Let $\cC$ be a $\cV$-enriched model category satisfying SM7. For $x \in \cC$ cofibrant and $y \in \cC$ fibrant, there is a canonical map of sets
    \begin{equation*}%
        [x,y] \to [\unit,\ul{\cC}(x,y)]
    \end{equation*}
    where $[x,y]$ denotes the hom-set in $\Ho(\cC)$ and the right side denotes the hom-set in $\Ho(\cV)$.
\end{lem}

\begin{proof}
    Since $\cC$ is the underlying category of the $\cV$-enriched category $\ul{\cC}$, we have:
    \[
        \xymatrix @R-0.5pc {
            \cC(x, y) \ar[d]_-{\text{quotient}} \ar[r]^-{\cong} & \cV(\unit, \underline{\cC}(x,y)) \ar[dd]^{\ga} \ar[r]^-{q^*} & \cV(Q\unit, \underline{\cC}(x,y)) \ar[d]^-{\text{quotient}} \\
    \cC(x,y)/\!\sim \ar[d]_{\cong} & & \cV(Q\unit, \underline{\cC}(x,y))/\!\sim \ar[d]^-\cong \\
    [x,y] \ar@{-->}[r]^-{\textcolor{red}{?}} & [\unit, \underline{\cC}(x,y)] \ar[r]^-{q^*}_-{\cong} & [Q\unit, \underline{\cC}(x,y)].
        }
    \]
    The proof of Proposition~\ref{pr:GeneGJ09PropII.3.10} shows that the dotted arrow is well-defined. If two maps $f,g \colon x \to y$ in $\cC$ are homotopic, then the corresponding maps $f',g' \colon \unit \to \ul{\cC}(x,y)$ in $\cV$ are (right) homotopic.
\end{proof}

The following statement generalizes the one for simplicial model categories from \cite{GoerssJ09}*{Lemma~II.3.15}. %

\begin{prop}\label{pr:hmtpYoneda}
    Let $\cC$ be a $\cV$-enriched model category satisfying SM7. Let $f \colon a \to b$ be a map in $\cC$ between cofibrant objects. 
    If $f$ is a weak equivalence, then for any fibrant object $z \in \cC$, the induced map on hom objects $f^* \colon \underline{\cC}(b,z) \to \underline{\cC}(a,z)$ is a weak equivalence in $\cV$.
\end{prop}

We will investigate the converse implication as a ``detection property'' below.

\begin{proof}
By \cite{GoerssJ09}*{Lemma~II.8.4}, %
the map $f \colon a \to b$ between cofibrant objects in $\cC$ has a factorization
\begin{equation*}
    \xymatrix{
    & x \ar[d]^-q \\
    a \ar[ur]^-j \ar[r]_-f & b 
    }
\end{equation*}
such that $j$ is a cofibration and $q$ is a left inverse to a trivial cofibration $i \colon b \to x$. 
If $f \colon a \to b$ is a weak equivalence, then the map $j \colon a \to x$ is a trivial cofibration, and hence induces a trivial fibration $j^* \colon \underline{\cC}(x,z) \to \underline{\cC}(a,z)$ for any fibrant object $z$. Similarly, the trivial cofibration $i$ induces a trivial fibration $i^*$, so that the map $q^* \colon \underline{\cC}(b,z) \to \underline{\cC}(x,z)$ is a weak equivalence. Therefore, $f^* = j^*q^*$ is a weak equivalence.
\end{proof}

\begin{prop}\label{pr:ImplicationsUnit}
    Let $\cC$ be a $\cV$-enriched model category satisfying SM7 and weakly tensored over~$\cV$. Consider the following conditions.
 \begin{enumerate}
    \item \textbf{$\pi_0$ of mapping space}: For any cofibrant $x \in \cC$ and fibrant $y \in \cC$, the canonical map $[x,y] \to [\unit,\underline{\cC}(x,y)]$ is a bijection.
    \item \label{item:detprop} \textbf{Detection property}: If a map $f \colon x \to y$ between cofibrant objects in $\cC$ is such that the restriction map $f^* \colon \underline{\cC}(y,z) \ral{\sim} \underline{\cC}(x,z)$ is a weak equivalence in $\cV$ for all fibrant object $z \in \cC$, then $f$ is a weak equivalence in $\cC$.
    \item \textbf{External unit axiom}: For any cofibrant replacement $q \colon Q\unit \ral{\sim} \unit$ and cofibrant object $x \in \cC$, the composite $x \otimes Q\unit \ral{x \ot q} x \ot \unit \ral{\rho_x} x$ is a weak equivalence in $\cC$.
\end{enumerate}
The implications (1) $\implies$ (2) $\implies$ (3) hold. If moreover the weak tensoring on $\cC$ is set-compatible, then all three conditions become equivalent.
\end{prop}
\begin{proof}
    \textbf{(1) $\implies$ (2)}. Let $f \colon a \to b$ be a map in $\cC$ between cofibrant objects such that 
\begin{equation}\label{eq:fstart}
    f^* \colon \underline{\cC}(b,z)\to \underline{\cC}(a,z) 
\end{equation}
is a weak equivalence in $\cV$ for all fibrant object $z \in \cC$. %
Applying $[\unit,-]$ to $\eqref{eq:fstart}$, we obtain the following commutative diagram of hom-sets:
\begin{equation*}
    \xymatrix{
       [\unit, \underline{\cC}(b,z)] \ar[r]^-{(f^*)_*}_-{\cong} & [\unit, \underline{\cC}(a,z)] \\
       [b,z] \ar[u]^-\cong \ar[r]^-{f^*}_-{\blue{\therefore \, \cong}} & [a,z] \ar[u]_-\cong
       }
\end{equation*}
where the vertical maps are bijections by the $\pi_0$ of mapping space assumption. Hence, the map $f \colon a \to b$ %
is a weak equivalence, by Proposition~\ref{pr:Saturation}.

\textbf{(2) $\implies$ (3)}. For cofibrant $x \in \cC$ and fibrant $y \in \cC$, consider the commutative diagram in $\cV$:
\begin{equation}\label{eq:WeakTensoring}
    \xymatrix{
    \ul{\cC}(x,y) \ar[dr]_{\cong} \ar[r]^-{\rho_x^*} & \ul{\cC}(x \ot \unit,y) \ar[d]^-{\phy} \ar[r]^-{(x \otimes q)^*} & \ul{\cC}(x\otimes Q\unit, y) \ar[d]^{\phy}_{\sim} \\
    & \ul{\cV}(\unit, \ul{\cC}(x,y)) \ar[r]^-{q^*}_-{\sim} & \ul{\cV}(Q\unit, \ul{\cC}(x,y)) \\
    }
\end{equation}
where the bottom map is a weak equivalence by the unit axiom %
in $\cV$. By 2-out-of-3, the top composite $\ul{\cC}(x,y) \to \ul{\cC}(x \ot Q\unit,y)$ is a weak equivalence. Since we assumed the detection property, the map $Q\unit \ot x \to x$ is a weak equivalence in $\cC$. 

\textbf{Set-compatible case: (3) $\implies$ (1)}. From diagram~\eqref{eq:WeakTensoring}, taking underlying sets %
and going to the homotopy categories yields a commutative diagram of hom-sets:
\[
    \xymatrix{
    [x,y] \ar[dr]_{\blue{\therefore \, \cong}} \ar[r]^-{\rho_x^*} & [x \ot \unit,y] \ar[d] \ar[r]^-{(x \otimes q)^*} & [x\otimes Q\unit, y] \ar[d]^{\cong} \\
    & [\unit, \ul{\cC}(x,y)] \ar[r]^-{q^*}_-{\cong} & [Q\unit, \ul{\cC}(x,y)] \\
    }
\]
By Proposition~\ref{pr:GeneGJ09PropII.3.10}, the vertical map on the right is a bijection $[x\otimes Q\unit, y] \cong [Q\unit, \ul{\cC}(x,y)]$. Since we assumed the unit axiom, the top composite is a bijection, hence the map $[x,y] \ral{\cong} [\unit, \ul{\cC}(x,y)]$ is a bijection.
\end{proof}

For reference, we include the dual statement.

\begin{prop}
    Let $\cC$ be a $\cV$-enriched model category satisfying SM7 and weakly cotensored over~$\cV$. Consider the following conditions.
 \begin{enumerate}
    \item \textbf{$\pi_0$ of mapping space}: For any cofibrant $x \in \cC$ and fibrant $y \in \cC$, the canonical map $[x,y] \to [\unit,\underline{\cC}(x,y)]$ is a bijection.
    \item \textbf{Detection property, dual form}: If a map $f \colon x \to y$ between fibrant objects in $\cC$ is such that the induced map $f_* \colon \underline{\cC}(a,x) \ral{\sim} \underline{\cC}(a,y)$ is a weak equivalence in $\cV$ for all cofibrant object $a \in \cC$, then $f$ is a weak equivalence in $\cC$.
    \item \textbf{External unit axiom, dual form}: For any cofibrant replacement $q \colon Q\unit \ral{\sim} \unit$ and fibrant object $y \in \cC$, the composite $y \ral{\eta_y} y^{\unit} \ral{q^*} y^{Q\unit}$ is a weak equivalence in $\cC$.
\end{enumerate}
The implications (1) $\implies$ (2) $\implies$ (3) hold. If moreover the weak cotensoring on $\cC$ is set-compatible, then all three conditions become equivalent.
\end{prop}

\begin{prop}\label{pr:unitpresev}
    Let $F \colon \cW \ral \cV \colon G$ be a weak monoidal Quillen adjunction such that the map $\ep \colon F(\unit_{\cW}) \ral{\cong} \unit_{\cV}$ is an isomorphism.
    \begin{enumerate}
        \item \label{item:unitpresev1} The change of enrichment along $G$ preserves the external unit axiom.
        \item \label{item:unitpresev2} The change of enrichment along $G$ preserves the $\pi_0$ of mapping space axiom.
        \item \label{item:unitpresev3} If $G$ reflects weak equivalences between fibrant objects, then the change of enrichment along $G$ preserves the detection property.
    \end{enumerate}
\end{prop}

\begin{proof} 
\Eqref{item:unitpresev1} Assume that the weak $\cV$-tensoring on $\cC$ satisfies the external unit axiom, that is, the composite $x \otimes Q\unit \ral{x \ot q} x \ot \unit \ral{\rho_x} x$  is a weak equivalence for $x$ cofibrant. For the weak $\cW$-tensoring on $G_*\cC$, the composite
\begin{equation*}
    \xymatrix{
    x\otimes Q\unit_{\cW} := x\otimes F(Q\unit_{\cW}) \ar[r]^-{x\otimes \Tilde{q}} & x\otimes \unit_{\cV} \ar[r]^-{\rho_x} & x
    }
\end{equation*}
is a weak equivalence since $\Tilde{q}$ is the weak equivalence $F(Q\unit_{\cW}) \ral{F(q)} F(\unit_{\cW}) \ral{\varepsilon} \unit_{\cV}$, %
making $F(Q\unit_{\cW})$ a cofibrant replacement of $\unit_{\cV}$.

\Eqref{item:unitpresev2} Let $x,y \in \cC$, with $x$ cofibrant and $y$ fibrant. By assumption, $\cC$ satisfies the $\pi_0$ of mapping space axiom, which yields:
        \begin{align*}
            [x,y] &\cong [\unit_\cV,\underline{\cC}(x,y)] \\
                  &\xrightarrow[\cong]{\Tilde{q}^*} [F(Q\unit_\cW),\underline{\cC}(x,y)],\ \text{since } \Tilde{q} \text{ is a weak equivalence} \\
                  &\cong [\mathbf{L}F(\unit_\cW),\underline{\cC}(x,y)] \\
                  &\cong [\unit_\cW,\mathbf{R}G\underline{\cC}(x,y)],\ \text{by the derived adjunction} \\
                  &\cong [\unit_\cW,G\underline{\cC}(x,y)],\ \text{since }\underline{\cC}(x,y)\text{ is fibrant} \\
                  &= [\unit_\cW,\underline{G_*\cC}(x,y)].   
        \end{align*}
\Eqref{item:unitpresev3} Assume that $G$ reflects weak equivalences between fibrant objects and that $\cC$ satisfies the detection property. Let $f \colon a \to b$ be in $(G_*\cC)_0 = \cC_0$, with cofibrant objects $a,b \in \cC$. Assume that $Gf^* \colon \underline{G_*\cC}(b,z) \xrightarrow[]{\sim} \underline{G_*\cC}(a,z)$, is a weak equivalence for all fibrant $z \in \cC$. 
Since $G$ reflects weak equivalences between fibrant objects, 
the restriction map $f^* \colon \underline{\cC}(b,z) \to \underline{\cC}(a,z)$ is a weak equivalence in $\cV$. Hence the map $f$ is a weak equivalence, as the $\cV$-enriched model category $\cC$ satisfies the detection property. Therefore, the category $G_*\cC$ satisfies the detection property also.
\end{proof}

\begin{ex}
    The following are examples of right Quillen functors $G\colon\cD\to \cC$ that reflect weak equivalences between fibrant objects.
    \begin{enumerate}
        \item Quillen equivalences \cite{Hovey99}*{Corollary~1.3.16}. %
        \item Let $\cC$ be a model category and $G \colon \cD \to \cC$ a right adjoint. If $\cD$ admits the right-induced model structure along $G$, then $G$ reflects weak equivalences.
    \end{enumerate}
\end{ex}

The following proposition gives an analogue of Proposition~\ref{pr:Hocot} for a weak (co)tensoring. %

\begin{prop}\label{pr:PropHoenrich}
    Let $\cC$ be a $\cV$-enriched model category satisfying SM7 and the $\pi_0$ of mapping space axiom.
    \begin{enumerate}
        \item \label{item:HoEnrich} The homotopy category $\Ho(\cC)$ inherits a $\Ho(\cV)$-enrichment. %
        \item \label{item:HoTensor} If the category $\cC$ is weakly tensored over $\cV$, then $\Ho(\cC)$ inherits a tensoring over $\Ho(\cV)$. 
        \item \label{item:HoCotensor} If the category $\cC$ is weakly cotensored over $\cV$, then $\Ho(\cC)$ inherits a cotensoring over $\Ho(\cV)$.
    \end{enumerate}
\end{prop}

\begin{proof} 
The proof is similar to \cite{Riehl14}*{Theorem~10.2.12}. Here we adapt the argument to a weak (co)tensoring. 

    \Eqref{item:HoEnrich} %
The enrichment of $\Ho(\cC)$ over %
$\Ho(\cV)$ is given by the total derived functor
    \begin{align}
        \mathbf{R}\underline{\cC}\colon\Ho(\cC)^\mathrm{op}\times \Ho(\cC) & \to \Ho(\cV) \nonumber\\
        \mathbf{R}\underline{\cC}(x,y) & = \gamma \underline{\cC}(Qx,Ry) \label{Hoenrichment},
    \end{align}
    where $\gamma \colon \cV \to \Ho(\cV)$ is the localization functor. %
    The underlying set of the hom-object in $\Ho(\cV)$ is
    \begin{flalign*}
    && U \mathbf{R}\underline{\cC}(x,y) &= [\unit, \mathbf{R}\underline{\cC}(x,y)] && \\
    && &\cong [\unit, \underline{\cC}(Qx,Ry)] && \\
    && &\cong [Qx,Ry] &&\text{by the $\pi_0$ of mapping space axiom} \\
    && &\cong [x,y]. &&
    \end{flalign*}
    Hence the underlying category of the $\Ho(\cV)$-enriched category $\Ho(\cC)$ coincides with the usual homotopy category $\Ho(\cC)$. 
    
\Eqref{item:HoTensor} %
    Applying the functor $\gamma \colon \cV \to \Ho(\cV)$ to the weak tensoring %
    \begin{equation*}
		\varphi_{v,x,y}\colon\underline{\cC}(x\otimes v,y)\to \underline{\cV}(v,\underline{\cC}(x,y))
	\end{equation*}
    yields a natural map in $\Ho(\cV)$
    \begin{equation}\label{eq:enrichediso}
		\gamma \varphi_{v,x,y} \colon \gamma \underline{\cC}(x\otimes v,y) \to \gamma \underline{\cV}(v,\underline{\cC}(x,y))
	\end{equation}
	which is an isomorphism for cofibrant $v \in \cV$, cofibrant $x \in \cC$, and fibrant $y \in \cC$. %
By Lemma~\ref{lem:cofibcoten}, the tensoring of cofibrant objects is cofibrant and the hom object from cofibrant to fibrant objects is fibrant. Using this, let us show that the total derived functor $\otimes^{\mathbf{L}}$ of the weak tensoring of $\cC$ over $\cV$ defines a tensoring of $\Ho(\cC)$ over $\Ho(\cV)$. For $v \in \cV$ and $x,y \in \cC$, we have:
    \begin{align*}
        \mathbf{R}\underline{\cC}(x\otimes^{\mathbf{L}} v,y) &\cong \mathbf{R}\underline{\cC}(Qx \otimes Qv,y) \\
                                                  &\cong \ga \underline{\cC}(Qx \otimes Qv,Ry) \\
                                                  &\cong \ga \underline{\cV}(Qv, \underline{\cC}(Qx,Ry)) \ \text{by } \eqref{eq:enrichediso} \\
                                                  &\cong \mathbf{R}\underline{\cV}(v,\mathbf{R}\underline{\cC}(x,y)).
    \end{align*}
    The naturality is given by the naturality of $\mathbf{R}\underline{\cV}$ and $\mathbf{R}\underline{\cC}$.
        
\Eqref{item:HoCotensor} The proof for the cotensoring over $\Ho(\cV)$ is similar to the one of tensoring.
\end{proof}

\section{Preservation of the enriched model structure}\label{sec:Pems}

Collecting the ingredients from the previous sections, 
we now define 
a weak $\cV$-model category, which is like %
a $\cV$-model category where the (co)tensoring is replaced by a weak (co)tensoring.

\begin{defn}\label{def:wvmod}
   Let $\cV$ be a closed symmetric monoidal model category. A \Def{weak $\cV$-model category} is a $\cV$-enriched model category $\cC$ weakly tensored and %
   cotensored over $\cV$, satisfying SM7 and the $\pi_0$ of mapping space axiom. %
\end{defn}

\begin{rem}\label{rem:wm-m}
    A $\cV$-model category $\cC$ is in particular a weak $\cV$-model category.
\end{rem}

\begin{thm}\label{thm:maintheo}
	Let $F \colon \cW \rla \cV \colon G$ be a weak monoidal Quillen adjunction such that $\ep \colon F(\unit_{\cW}) \ral{\cong} \unit_{\cV}$ is an isomorphism. If $\cC$ is a weak $\cV$-model category, then $G_*\cC$ is a weak $\cW$-model category. 
    
    If moreover the weak (co)tensoring on $\cC$ is set-compatible (resp.\ unital), then so is the weak (co)tensoring on $G_*\cC$. 
    Those conditions hold in particular when $\cC$ is a $\cV$-model category.
\end{thm}

\begin{proof} 
The  category $G_*\cC$ is a $\cW$-enriched category since $G$ is a lax monoidal functor. Also, $G_*\cC$ has the same underlying category as $\cC$ by Lemma~\ref{lem:sameunderlying}, since \mbox{$F(\unit_{\cW})\cong \unit_{\cV}$}. In particular, the underlying category of $G_*C$ still has a model structure. The category $G_*\cC$ satisfies SM7 by Lemma~\ref{lem:G*SM7}, and inherits a weak tensoring and weak cotensoring over $\cW$ by Proposition~\ref{pr:PropW_ten_coPreserv}. Finally, the category $G_*\cC$ satisfies the %
$\pi_0$ of mapping space axiom 
by Proposition~\ref{pr:unitpresev}.
\end{proof}

\begin{rem}
By Proposition~\ref{pr:unitpresev}, the analogue of Theorem~\ref{thm:maintheo} where we replace the $\pi_0$ of mapping space axiom with the external unit axiom is also true. If moreover $G$ reflects weak equivalences between fibrant objects, we could replace the $\pi_0$ of mapping space axiom with the detection property.
\end{rem}

\begin{ex}
Revisiting Example~\ref{ex:DGcat}, let $R$ be a commutative ring and $\cC$ a DG-category over $R$, i.e., a category enriched over $\Ch(R)$. The underlying %
homotopy theory of $\cC$ is obtained by changing the enrichment along the right adjoints displayed at the bottom:
\[
    \xymatrix{
        s\Set \ar@<+0.6ex>[r]^-{R(-)} & s\Mod{R} \ar@<+0.6ex>[l]^-{U} \ar@<+0.6ex>[r]^-{N} & \Ch_{\geq 0}(R) \ar@<+0.6ex>[l]^-{\Ga} \ar@<+0.6ex>[r]^-{\iota} & \Ch(R) \ar@<+0.6ex>[l]^-{\tau_{\geq 0}} \\
    }
\]
Since the $s\Set$-enriched category $U_* \Ga_* (\tau_{\geq 0})_* \cC$ has mapping spaces that are %
Kan complexes, it models an $\infty$-category \cite{Bergner07}. The underlying $\infty$-category of $\cC$ can also be constructed via the DG-nerve \cite{Lurie17}*{\S 1.3.1}. 
The middle adjunction is a weak monoidal Quillen equivalence, whereas the first and third adjunctions are strong monoidal Quillen pairs. If $\cC$ was in fact a DG \emph{model} category, i.e., a $\Ch(R)$-model category, then $\Ga_* (\tau_{\geq 0})_* \cC$ is a weak $s\Mod{R}$-model category and $U_* \Ga_* (\tau_{\geq 0})_* \cC$ is a weak $s\Set$-model category, by Theorem~\ref{thm:maintheo}.
\end{ex}

%
\begin{comment}
\begin{ex}
Revisiting Example~\ref{ex:DGcat}, let $R$ be a commutative ring and $\cC$ a DG-category over $R$, i.e., a category enriched over $\Ch(R)$. The underlying %
homotopy theory of $\cC$ is obtained by changing the enrichment along the functors 
\[
    \xymatrix{
        \Ch(R) \ar[r]^-{\tau_{\geq 0}} & \Ch_{\geq 0}(R) \ar[r]^-{\Ga} & s\Mod{R} \ar[r]^-{U} & s\Set. \\
    }
\]
Since the $s\Set$-enriched category $U_* \Ga_* (\tau_{\geq 0})_* \cC$ has mapping spaces that are %
Kan complexes, it models an $\infty$-category \cite{Bergner07}. The underlying $\infty$-category of $\cC$ can also be constructed via the DG-nerve \cite{Lurie17}*{\S 1.3.1}. Consider the three adjunctions with the left adjoints displayed on top:
\[
    \xymatrix{
        s\Set \ar@<+0.8ex>[r]^-{R(-)} & s\Mod{R} \ar@<+0.8ex>[l]^-{U} \ar@<+0.8ex>[r]^-{N} & \Ch_{\geq 0}(R) \ar@<+0.8ex>[l]^-{\Ga} \ar@<+0.8ex>[r]^-{\iota} & \Ch(R) \ar@<+0.8ex>[l]^-{\tau_{\geq 0}} \\
    }
\]
The middle adjunction is a weak monoidal Quillen equivalence, whereas the first and third adjunctions are strong monoidal Quillen pairs. If $\cC$ was in fact a DG model category, i.e., a $\Ch(R)$-model category, then $\Ga_* (\tau_{\geq 0})_* \cC$ is a weak $s\Mod{R}$-model category and $U_* \Ga_* (\tau_{\geq 0})_* \cC$ is a weak $s\Set$-model category, by Theorem~\ref{thm:maintheo}.
\end{ex}
\end{comment}
%

\section{Application: Comparing enrichments}\label{sec:App}

Via the Dold--Kan correspondence, we can produce two enrichments of $\mathrm{Ch}_{\geq 0}(R)$ over $s\Mod{R}$, for a commutative ring $R$: one by applying Dold--Kan globally to the whole category, the other by applying Dold--Kan locally to each hom complex.

\begin{defn} 
	The \Def{``global enrichment''} and \Def{``local enrichment''} of $\mathrm{Ch}_{\geq 0}(R)$ over $s\Mod{R}$ are defined by
		\begin{align*}
			\HOM^{\mathrm{global}}_{\Ch_{\geq 0}(R)}(C,D) &:= \HOM_{s\Mod{R}}(\Gamma(C),\Gamma(D)) = \Ga(D)^{\Ga(C)} \\
			\HOM^{\mathrm{local}}_{\Ch_{\geq 0}(R)}(C,D) &:= \Gamma \tau_{\geq 0}\HOM_{\Ch(R)}(C,D) = \Gamma \HOM_{\Ch_{\geq 0}(R)}(C,D) = \Ga(D^C)
		\end{align*}
for all $C,D \in \Ch_{\geq 0}(R)$. 
	Here, the bifunctors $\HOM_{s\Mod{R}}$ and $\HOM_{\Ch_{\geq 0}(R)}$ are the respective internal hom's.
\end{defn}

\begin{prop}\label{pr:TALBOT} 
\begin{enumerate}
	\item \label{item:GlobalTensor} Considering the ``global enrichment''\footnote{The ``global enrichment'', with the tensoring and cotensoring defined in part~\Eqref{item:GlobalTensor} of Proposition~\ref{pr:TALBOT}, is the one referred to on page~3 of \href{https://math.mit.edu/events/talbot/2012/2012TalbotExercises.pdf}{Background material for TALBOT 2012:  https://math.mit.edu/events/talbot/2012/2012TalbotExercises.pdf}.}of $\Ch_{\geq 0}(R)$ over $s\Mod{R}$, the formulas
\[
C \otimes A := N(\Gamma(C) \otimes A) \qquad \text{and} \qquad C^A := N\HOM_{s\Mod{R}}(A,\Gamma(C)) = N(\Gamma(C)^A)
\]
		define a tensoring and a cotensoring on $\Ch_{\geq 0}(R)$ over $s\Mod{R}$, for any $C \in \Ch_{\geq 0}(R)$ and $A \in s\Mod{R}$.
	\item \label{item:LocalWeakTensor} For the ``local enrichment'' of $\Ch_{\geq 0}(R)$ over $s\Mod{R}$, the formulas
\[
C \otimes A := C \ot N(A) \qquad \text{and} \qquad C^A := \HOM_{\Ch_{\geq 0}(R)}(N(A),C) = C^{N(A)}
\]
        define a weak tensoring and a weak cotensoring on $\Ch_{\geq 0}(R)$ over $s\Mod{R}$, for any $C \in \Ch_{\geq 0}(R)$ and $A \in s\Mod{R}$.
	\item \label{item:LocalNoTensor} There is neither a tensoring nor a cotensoring on $\Ch_{\geq 0}(R)$ over $s\Mod{R}$ associated to the ``local enrichment''.
\end{enumerate}
\end{prop}

\begin{proof}
Part~\Eqref{item:GlobalTensor} is an instance of Proposition~\ref{pr:EquivCoTen}. 
Part~\Eqref{item:LocalWeakTensor} follows from Proposition~\ref{pr:PropW_ten_coPreserv} applied to the change of enrichment along $\Ga \colon \Ch_{\geq 0}(R) \to s\Mod{R}$, using Example~\ref{ex:DK-WeakMonoidal}. %
Part~\Eqref{item:LocalNoTensor} follows from Proposition~\ref{pr:coten-strong}, since the normalization $N \colon s\Mod{R} \to \Ch_{\geq 0}(R)$ is not strong monoidal, cf.\ Example~\ref{ex:DGcat}.
\end{proof}

Nevertheless, %
Opadotun \cite{Opadotun21}*{Theorem~6.3.1} proved that the ``global enrichment'' and ``local enrichment'' of $\Ch_{\geq 0}(R)$  over $s\Mod{R}$ %
are homotopy equivalent, more precisely part~\Eqref{item:SimplicialEnrichment} here: 

\begin{thm}\label{thm:Michael}
	Let $R$ be a commutative ring $R$.
    \begin{enumerate}
        \item \label{item:SimplicialEnrichment} For all chain complexes $C,D \in \Ch_{\geq 0}(R)$, the natural map of simplicial $R$-modules 
	\begin{equation*}
		\xymatrix @R-1pc {
		\AW^* \colon \Hom_{\Ch_{\geq 0}(R)}(C \otimes N(R \Delta^\bullet), D) \ar@{=}[d] \ar[r] & \Hom_{\Ch_{\geq 0}(R)}(N(\Gamma C\otimes R \Delta^\bullet), D) \ar@{=}[d] \\
		\Ga(D^C) & \Ga(D)^{\Ga(C)} \\
		}
	\end{equation*}
	is a homotopy equivalence, with homotopy inverse $\EZ^*$.
	 	\item \label{item:ChainEnrichment} For all simplicial modules $A,B \in s\Mod{R}$, the natural map of chain complexes $\AW^* \colon N(B)^{N(A)} \to N(B^A)$ is a chain homotopy equivalence, with homotopy inverse $\EZ^*$.
    \end{enumerate}
\end{thm}
Note that $\EZ^*$ is also a retraction to $\AW^*$, i.e., $\EZ^* \circ \AW^* = (\AW \circ \EZ)^* = \id$, so that $\Ga(D^C)$ is a deformation retract of $\Ga(D)^{\Ga(C)}$ 
and likewise $N(B)^{N(A)}$ is a deformation retract of $N(B^A)$.

The argument is also given in \cite{Ngopnang24hur}*{Proposition~5.2}, in the form of part~\Eqref{item:ChainEnrichment}. %
Note that parts~\Eqref{item:SimplicialEnrichment} and \Eqref{item:ChainEnrichment} are equivalent, by Theorem~\ref{thm:DoldKan}\Eqref{item:Equivalence}\Eqref{item:Homotopy}.

As an application of %
weak (co)tensorings, we obtain the following.

\begin{proof}[Alternate proof of Theorem~\ref{thm:Michael}]
Consider $\Ga \colon \Ch_{\geq 0}(R) \to s\Mod{R}$ with the Hurewicz model structures, where weak equivalences are the homotopy equivalences \cite{Ngopnang24hur}*{Proposition~3.3}. Change of enrichment of the self-enriched model category $\Ch_{\geq 0}(R)$ along $\Gamma$ yields the weak $s\Mod{R}$-model category $\Ga_* \Ch_{\geq 0}(R)$. %
The weak cotensoring 
is a map 
    \begin{equation}\label{DKwcoten}
        \tild{\psi}_{C,E,D} \colon \Gamma \HOM_{\Ch_{\geq 0}(R)}(E,D^C) \to \HOM_{s\Mod{R}}(\Gamma C, \Gamma \HOM_{\Ch_{\geq 0}(R)}(E,D))
    \end{equation}
    in $s\Mod{R}$ which is a weak equivalence when $C,E\in \Ch_{\geq 0}(R)$ are cofibrant and \mbox{$D \in \Ch_{\geq 0}(R)$} is fibrant. But in the Hurewicz model structure, every object is cofibrant and fibrant, so that the map~\eqref{DKwcoten} is a homotopy equivalence for all $C,D,E$. Specializing to $E = \unit_{\Ch_{\geq 0}(R)} = R[0]$, \eqref{DKwcoten} becomes 
    \begin{align*}
         \tild{\psi}_{C,\unit,D} \colon \Gamma \HOM_{\Ch_{\geq 0}(R)}\left(R[0],D^C\right) & \ral{\simeq} \HOM_{s\Mod{R}} \left(\Gamma C, \Gamma \HOM_{\Ch_{\geq 0}(R)} \left( R[0],D \right) \right) \\ 
                                    & \cong \HOM_{s\Mod{R}} \left(\Gamma C, \Gamma D \right),
    \end{align*}
    i.e, $\Gamma(D^C) \simeq \Gamma(D)^{\Gamma(C)}$. Inspecting the formula for the weak cotensoring $\tild{\psi}_{C,E,D}$ in Proposition~\ref{pr:PropW_ten_coPreserv} confirms that the map $\tild{\psi}_{C,\unit,D}$ is the restriction $\AW^*$ along Alexander--Whitney.

For the analogous homotopy equivalence $N(B)^{N(A)} \simeq N(B^A)$, apply the same argument to the change of enrichment along $N \colon s\Mod{R} \to \Ch_{\geq 0}(R)$.
\end{proof}

\begin{rem}
For a non-commutative ring $R$, the statements in this section are similar, using the enrichment, tensoring, and cotensoring of $s\Mod{R}$ over $s\Ab$ and of $\Ch_{\geq0}(R)$ over $\Ch_{\geq0}(\Z)$, as in Lemma~\ref{lem:NonCommutative}. %
\end{rem}

\section*{Part II. Sources of examples}\label{sec:Examples}

To expand the applicability of weak monoidal Quillen pairs, it would be useful to have a large supply of examples. In \cite{SchwedeS03equ}*{Propositions~3.16, 3.17}, Schwede and Shipley give sufficient conditions for a Quillen pair to be weak monoidal. Here we give some constructions to produce new weak monoidal Quillen pairs from old. We will look at diagram categories, equivariant homotopy, and Bousfield localizations.

\section{Diagram categories}

\begin{prop}\label{pr:DiagramCat}
Let $F \colon \cC \rla \cD \colon G$ be a %
Quillen adjunction and let $K$ be a small category. Consider the pointwise monoidal structures on $\cC^K$ and $\cD^K$. %
\begin{enumerate}
    \item The induced adjunction on diagram categories
\[
F^K \colon \cC^K \rla \cD^K \colon G^K
\]
is also a 
Quillen adjunction with respect to each of these model structures on diagram categories when they exist: %
projective, injective, or Reedy (if $K$ is a Reedy category).
    \item If moreover the original adjunction $F \dashv G$ is a Quillen equivalence, then so is the induced adjunction $F^K \dashv G^K$.
    \item If $\cC$ and $\cD$ are monoidal and $F \colon \cC \rla \cD \colon G$ is a weak monoidal Quillen adjunction, then so is the induced adjunction $F^K \dashv G^K$.
\end{enumerate}
\end{prop}

\begin{proof}
For~(1) and (2), see \cite{Hirschhorn03}*{Theorem~11.6.5, Proposition~15.4.1} \cite{Lurie09htt}*{Remark~A.2.8.6}.

(3) \textbf{Lax monoidal:} The right adjoint $G^K$ is lax monoidal, using the lax monoidal structure of $G$ pointwise.

\textbf{Oplax comultiplication:} The oplax monoidal structure of $F^K$ is that of $F$ pointwise. We want to show that the oplax monoidal transformation
\begin{equation}\label{eq:Comult}
\de \colon F^K(x \ot y) \to F^K(x) \ot F^K(y)
\end{equation}
is a weak equivalence in $\cD^K$ (i.e., a pointwise weak equivalence) when $x$ and $y$ are cofibrant in $\cC^K$. In all three cases (projective, Reedy, injective), being cofibrant in $\cC^K$ implies being pointwise cofibrant (which is the same as injective cofibrant). Thus in position $k \in K$, the map of $K$-shaped diagrams \eqref{eq:Comult} becomes
\[
\de \colon F(x_k \ot y_k) \ral{\sim} F(x_k) \ot F(y_k),
\]
which is a weak equivalence in $\cD$ by assumption on $F$.

\textbf{Unit condition:} %
The tensor unit $\unit_{\cC^K}$ in $\cC^K$ is pointwise, i.e., the constant diagram $c(\unit_{\cC})$. In all three cases (projective, Reedy, injective), a cofibrant replacement $q \colon Q \unit_{\cC^K} \ral{\sim} \unit_{\cC^K}$ is in particular pointwise cofibrant. Therefore, for all $k \in K$, the $k$\textsuperscript{th} component 
$q_k \colon (Q \unit_{\cC^K})_k \ral{\sim} \unit_{\cC}$ 
is a cofibrant replacement of $\unit_{\cC}$ in $\cC$. By assumption on $F$, the composite
\[
\xymatrix{
F (Q \unit_{\cC^K})_k \ar[r]^-{F(q_k)} & F(\unit_{\cC}) \ar[r]^-{\ep} & \unit_{\cD} \\
}
\]
is a weak equivalence in $\cD$. Hence the composite 
$F^K (Q \unit_{\cC^K}) \to F(\unit_{\cC^K}) \to \unit_{\cD^K}$ 
is a weak equivalence in $\cD^K$.
\end{proof}

Recall some sufficient conditions for the existence of model structures on a diagram category $\cC^K$.

\begin{prop}\label{pr:DiagramModel}
Let $\cC$ be a model category and $K$ a small category.
\begin{enumerate}
	\item The projective model structure on $\cC^K$ exists if the model category $\cC$ is cofibrantly generated \cite{Hirschhorn03}*{Theorem~11.6.1}.
	\item The injective model structure on $\cC^K$ exists if the model category $\cC$ is combinatorial \cite{Lurie09htt}*{Proposition~A.2.8.2}, 
	or more generally $\cC$ is locally presentable and has an accessible model structure \cite{HessKRS17}*{Theorem~3.4.1} \cite{GarnerKR20}. 
	
	For other generalizations, see \cite{BayehHKKRS15}*{\S 4.3} and \cite{Moser19}.
	\item The Reedy model structure on $\cC^K$ exists if $K$ is a Reedy category and $\cC$ is any model category \cite{Hovey99}*{Theorem~5.2.5} \cite{Hirschhorn03}*{Theorem~15.3.4} \cite{Lurie09htt}*{Proposition~A.2.9.19}.
\end{enumerate}
\end{prop}

\begin{rem}
If the small category $K$ is discrete, then the diagram category is a product $\cC^K \cong \prod_{k \in K} \cC$. In this case, both the projective and injective model structures exist and coincide with the product model structure. If moreover $\cC$ is a monoidal model category, then so is $\cC^K$. 
\end{rem}

\begin{rem}
    The $q$-model structure on $s\Mod{R}$ is combinatorial \cite{Balchin21}*{Proposition~4.4.4}, as is the $q$-model structure on $\Ch_{\geq 0}(R)$ \cite{Hovey99}*{Theorem~2.3.11} \cite{SchwedeS03equ}*{\S 4.1} \cite{MayP12}*{Theorem~18.4.2}. In general, the $h$-model structure on $\Ch_{\geq 0}(R)$ is not cofibrantly generated (in the unenriched sense) \cite{ChristensenH02}*{Corollary~5.12}, though it is cofibrantly generated in the $\Mod{R}$-enriched sense \cite{Riehl14}*{Examples~13.2.5, 13.4.5}.
\end{rem}

Now we recall some conditions for a diagram category to be a \emph{monoidal} model category.

\begin{prop}\label{pr:DiagramMonoidal}
Let $\cC$ be a monoidal model category and $K$ a small category. In each statement below, assume that $\cC^K$ admits the given model structure (projective, injective, or Reedy).
\begin{enumerate}
	\item In all three cases, the unit axiom holds in $\cC^K$.
	\item The injective model structure on $\cC^K$ satisfies the pushout-product axiom. 
	\item The projective model structure on $\cC^K$ satisfies the pushout-product axiom if $\cC$ is cofibrantly generated and $K$ has finite coproducts. %
\end{enumerate}
\end{prop}

\begin{proof}
Part~(1) holds since cofibrant objects in $\cC^K$ are in particular pointwise cofibrant. Part~(2) holds since cofibrations and weak equivalences are pointwise in the injective model structure. This was also observed in \cite{Barwick10}*{Proposition~4.51}. 
Part~(3) was observed by Yalin \cite{Yalin14}\footnote{The statement is not in the published version, but can be found in the arXiv preprint version~3, Lemma~3.8.} and Szumi{\l}o \cite{MO_Szumilo13}; we recall the argument for later use. If $I$ is a set of generating cofibrations in $\cC$, a set of generating cofibrations for the projective model structure on $\cC^K$ is given by 
\[
    \{ i \odot K(a,-) \mid i \in I, a \in K \}
\]
where $\odot \colon \cC \x \Set \to \cC$ denotes the canonical tensoring of $\cC$ over $\Set$ given by coproducts, i.e., $X \odot S = \coprod_{s \in S} X$. Take two generating cofibrations $i \odot K(a,-)$ and $j \odot K(b,-)$. Since the tensor product in $\cC$ preserves coproducts in each variable, we obtain the pushout-product
\begin{align*}
    \left( i \odot K(a,-) \right) \square \left( j \odot K(b,-) \right) &\cong (i \square j) \odot \left( K(a,-) \x K(b,-) \right) \\ 
    &\cong (i \square j) \odot K(a \amalg b,-), 
\end{align*}
which is a cofibration in $\cC^K$. The same argument works for acyclic cofibrations since weak equivalences in $\cC^K$ are pointwise.
\end{proof}

For conditions under which the Reedy model structure on $\cC^K$ %
is monoidal, see \cite{Barwick10}*{Theorem~3.51}. For alternate conditions making the projective model structure monoidal, see \cite{Barwick10}*{Proposition~4.52}.

\begin{rem}
The Dold--Kan correspondence has been used for various presheaf categories, for instance in \cite{RondigsO08}*{\S 2.2.4} and \cite{PavlovS19}*{\S 8}. 
\end{rem}

\section{Equivariant homotopy}

Throughout this section, let $G$ be a %
discrete group (possibly infinite) 
and let $\cO_G$ be its orbit category, which consists of transitive $G$-sets $G/H$ and $G$-equivariant functions $G/H \to G/K$ between them. We will work in the setup of Elmendorf's theorem \cite{Elmendorf83} \cite{May96}*{\S V.3, VI.6}; %
we follow the %
generalization in \cite{Stephan16}. 

\begin{defn}
For a model category $\cC$, consider the category of $G$-objects $\cC^G = \Fun(G,\cC)$, where $G$ is viewed as a one-object category. The \Def{fine model structure} on $\cC^G$ (if it exists) has as fibrations (resp.\ weak equivalences) the maps $f \colon X \to Y$ in $\cC^G$ such that the induced map on fixed points $f^H \colon X^H \to Y^H$ is a fibration (resp.\ weak equivalence) in $\cC$ for all subgroup $H \leq G$. 
\end{defn}

The fine model structure is also called the \emph{$\mathcal{F}$-model structure} for the family of subgroups $\mathcal{F} = \{ \text{all subgroups of } G \}$. %
The inclusion of the free orbit $i \colon \{G/e\} \inj \cO_G$ induces a Quillen pair
\begin{equation}\label{eq:Elmendorf}
i^* \colon \cC^{\cO_G^{\opp}} \rla \cC^G \colon i_*
\end{equation}
where $i^*$ is evaluation at the orbit $G/e$, and the right adjoint $i_*$ sends a $G$-object $X$ to its diagram of fixed point objects $\{ X^H \mid H \leq G \}$. We say that $\cC$ \Def{satisfies Elmendorf's theorem} if the Quillen pair~\eqref{eq:Elmendorf} is a Quillen equivalence. 
Sufficient conditions for the existence of the fine model structure are given in \cite{Stephan16}*{Proposition~2.6}.

\begin{lem}\label{lem:TrivialAction}
Let $\cC$ be a model category such that the fine model structure on $\cC^G$ exists.
\begin{enumerate}
    \item For any cofibrant object $X \in \cC$ and subgroup $H \leq G$, the $G$-object $X \odot G/H$ is cofibrant in $\cC^G$. This holds in particular for $X$ endowed with the trivial $G$-action, denoted $\Triv(X) = X \odot G/G$.
    \item If a map $f \colon X \to Y$ in $\cC$ is a fibration (resp.\ weak equivalence), then so is $\Triv(f) \colon \Triv(X) \to \Triv(Y)$ in $\cC^G$.    
\end{enumerate}
\end{lem}

\begin{proof}
Part~(1) follows from the fact that the functor $- \odot G/H \colon \cC \to \cC^G$ is left adjoint to the $H$-fixed points functor $(-)^H \colon \cC^G \to \cC$ \cite{Stephan16}*{Lemma~2.5}, and an acyclic fibration $p \colon Y \to Z$ in $\cC^G$ induces an acyclic fibration on fixed points $p^H \colon Y^H \to Z^H$. 
Part~(2) follows from the natural isomorphism $\eta \colon X \ral{\cong} \Triv(X)^H$ for any $X \in \cC$.
\end{proof}

\begin{lem}\label{lem:TensorTrivial}
Let $\cC$ be a cofibrantly generated monoidal model category that satisfies the ``cellularity conditions'' from  \cite{Stephan16}*{Proposition~2.6}, so that $\cC^G$ admits the fine model structure. 
\begin{enumerate}
    \item \label{item:CofibrantFixed} For any cofibrant $X \in \cC^G$ and subgroup $H \leq G$, the fixed points object $X^H$ is cofibrant in $\cC$.
    \item \label{item:TensorTrivial} For any cofibrant $X \in \cC^G$ and any object $Y \in \cC$, the natural map in $\cC$
\[
\xymatrix{
X^H \ot Y \ar[r]^-{\id \ot \eta}_-{\cong} & X^H \ot \Triv(Y)^H \ar[r]^-{\phy} & (X \ot \Triv(Y))^H \\
}
\]
is an isomorphism. Here $\phy \colon X^H \ot Y^H \to (X \ot Y)^H$ denotes the canonical map. 
\end{enumerate}
\end{lem}

\begin{proof}
Part~(1) was argued in the proof of \cite{Stephan16}*{Proposition~2.6}. 

For part~(2), 
the functors $- \ot Y \colon \cC \to \cC$ and $- \ot \Triv(Y) \colon \cC^G \to \cC^G$ preserve all (small) colimits. %
By the cellularity conditions~(i) and (ii), the fixed points functor $(-)^H \colon \cC^G \to \cC$ preserves the colimits that build the cell complex $X$ from cell attachments. Thus it suffices to check the case of a building block $X = A \odot G/K$ for some object $A \in \cC$ and subgroup $K \leq G$. That case follows from the cellularity condition~(iii):

\begin{align*}
    \left( X \ot \Triv(Y) \right)^H &= \left( (A \odot G/K) \ot \Triv(Y) \right)^H \cong \left( \coprod_{G/K} A \ot Y \right)^H \\
    &\cong \left( (A \ot Y) \odot G/K \right)^H \cong (A \ot Y) \odot (G/K)^H \quad \text{by condition (iii)} \\
    &\cong \coprod_{(G/K)^H} A \ot Y \cong \left( A \odot (G/K)^H \right) \ot Y \cong (A \odot G/K)^H \ot Y \quad \text{by condition (iii)} \\
    &= X^H \ot Y. \qedhere
\end{align*}
\end{proof}

The following argument was kindly suggested to us by David White \cite{MO_White26}; 
a variation can be found in \cite{MandellM02}*{Lemma~III.1.22}.

\begin{prop}\label{pr:GMonoidal}
Let $\cC$ be a cofibrantly generated monoidal model category. %
\begin{enumerate}
    \item The projective model structure on the presheaf category $\cC^{\cO_G^{\opp}}$ is a monoidal model structure.
\end{enumerate}
Now assume moreover that $\cC$ satisfies the ``cellularity conditions'' from  \cite{Stephan16}*{Proposition~2.6}.
\begin{enumerate}[resume]
    \item The fine model structure on $\cC^G$ is a monoidal model structure.
    \item The adjunction~\eqref{eq:Elmendorf} is a strong monoidal Quillen pair.
\end{enumerate}
\end{prop}

\begin{proof}
(1) We cannot invoke Proposition~\ref{pr:DiagramMonoidal} directly, since the orbit category $\cO_G$ doesn't have finite products. Nonetheless, as $G$-sets, a product of orbits decomposes as a disjoint union of $G$-orbits:
\[
G/H \x G/K \cong \coprod_{\al} G/L_{\al}.
\]
Hence a pushout-product of generating cofibrations in $\cC^{\cO_G^{\opp}}$
\[
\left( i \odot \cO_G(-,G/H) \right) \square \left( j \odot \cO(-,G/K) \right) \cong (i \square j) \odot \left( \cO_G(-,G/H) \x \cO_G(-,G/K) \right)
\]
becomes a coproduct of copies of $i \square j$ indexed by the presheaf of sets
\begin{align*}
    \cO_G(-,G/H) \x \cO_G(-,G/K) &= \Set^G(-,G/H) \x \Set^G(-,G/K) \\
    &\cong \Set^G(-,G/H \x G/K) \cong \Set^G(-, \coprod_{\al} G/L_{\al}) \\
    &\cong \coprod_{\al} \Set^G(-, G/L_{\al}) \qquad \text{since the input $(-)$ is a single orbit}\\
    &= \coprod_{\al} \cO_G(-, G/L_{\al}).
\end{align*}
Thus the pushout-product 
\[
\left( i \odot \cO_G(-,G/H) \right) \square \left( j \odot \cO(-,G/K) \right) \cong \coprod_{\al} (i \square j) \odot \cO_G(-,G/L_{\al})
\]
is a cofibration in $\cC^{\cO_G^{\opp}}$. The same argument works for acyclic cofibrations.

(2) \textbf{Pushout-product axiom.} The same argument as in part~(1) applies, using the generating cofibrations $i \odot G/H$ for the fine model structure on $\cC^G$.

\textbf{Unit axiom.} Let $q \colon Q\unit_{\cC} \ral{\sim} \unit_{\cC}$ be a cofibrant replacement of the tensor unit in $\cC$. By Lemma~\ref{lem:TrivialAction}, $\Triv(q) \colon \Triv(Q\unit_{\cC}) \to \Triv(\unit_{\cC}) = \unit_{\cC^G}$ is a cofibrant replacement in $\cC^G$. For $X \in \cC^G$ cofibrant, we want to show that the map
\[
\id \ot \Triv(q) \colon X \ot \Triv(Q\unit_{\cC}) \to X \ot \Triv(\unit_{\cC}) \cong X
\]
is a weak equivalence in $\cC^G$, i.e., a weak equivalence on all fixed points. For a subgroup $H \leq G$, consider the commutative diagram in $\cC$:
\[
\xymatrix @R-0.4pc {
X^H \ot Q\unit_{\cC} \ar[d]_{\id \ot \eta}^{\cong} \ar[r]^-{\id \ot q}_-{\blue{\sim}} & **[r] X^H \ot \unit_{\cC} \cong X^H \ar[d]^{\id \ot \eta}_{\cong}  \\
X^H \ot \Triv(Q\unit_{\cC})^H \ar[d]_{\phy}^{\blue{\cong}} \ar[r] & X^H \ot \Triv(\unit_{\cC})^H \ar[d]^{\phy}_{\cong} \\
\left( X \ot \Triv(Q\unit_{\cC}) \right)^H \ar[r]^-{(\id \ot \Triv(q))^H}_-{\blue{\therefore \, \sim}} & **[r] \left( X \ot \Triv(\unit_{\cC}) \right)^H \cong X^H. \\
}
\]
The top map is a weak equivalence by the unit axiom in $\cC$, since $X^H$ is cofibrant by Lemma~\ref{lem:TensorTrivial}\eqref{item:CofibrantFixed}. The map $\phy$ on the left is an isomorphism, by Lemma~\ref{lem:TensorTrivial}\eqref{item:TensorTrivial}. Hence the bottom map is a weak equivalence.

(3) The left adjoint $i^* \colon \cC^{\cO_G^{\opp}} \to \cC^G$ is strong monoidal, since the monoidal structure on both sides is pointwise, i.e., given presheaves $E,F \in \cC^{\cO_G^{\opp}}$, we have $(E \ot F)(G/e) \cong E(G/e) \ot F(G/e)$ in $\cC^G$.

\textbf{Unit condition.} Since the orbit category $\cO_G$ has a terminal object $G/G = \ast$, a constant presheaf $c(X)$ on a cofibrant object $X \in \cC$ is %
cofibrant in $\cC^{\cO_G^{\opp}}$. From a cofibrant replacement $q \colon Q\unit_{\cC} \ral{\sim} \unit_{\cC}$ in $\cC$, taking constant presheaves yields a cofibrant replacement in $\cC^{\cO_G^{\opp}}$
\[
c(q) \colon c(Q\unit_{\cC}) \ral{\sim} c(\unit_{\cC}) = \unit_{\cC^{\cO_G^{\opp}}}.
\]
Applying the left adjoint $i^*$ yields
\[
\xymatrix @R-0.6pc {
i^* c(Q\unit_{\cC}) \ar[d]_{\cong} \ar[r]^-{i^* c(q)} & i^* c(\unit_{\cC}) \ar[d]^{\cong} \\
\Triv(Q\unit_{\cC}) \ar[r]^-{\Triv(q)}_{\blue{\sim}} & \Triv(\unit_{\cC}), \\
}
\]
which is a weak equivalence in $\cC^G$, by Lemma~\ref{lem:TrivialAction}.
\end{proof}

\begin{prop}\label{pr:GObjects}
Let $\cC$ and $\cD$ be model categories that admit the fine model structure, and let $L \colon \cC \rla \cD \colon R$ be a Quillen pair.
\begin{enumerate}
    \item The induced adjunction on $G$-objects $L \colon \cC^G \rla \cD^G \colon R$ is also a Quillen pair.
    \item Assume that $L \colon \cC \rla \cD \colon R$ is a Quillen equivalence, and the left adjoint $L$ ``homotopically commutes with fixed points'', i.e., %
    for all cofibrant $G$-object $X \in \cC^G$, the canonical map $L(X^H) \to (LX)^H$ in $\cD$ is a weak equivalence for all subgroup $H \leq G$. Then the induced adjunction on $G$-objects $L \colon \cC^G \rla \cD^G \colon R$ is a Quillen equivalence.
    \item In the setup of part~(2), $\cC$ satisfies Elmendorf's theorem if and only if $\cD$ does.
    \item \label{item:TensorFixed} Assume that $\cC$ and $\cD$ are monoidal model categories, the left adjoint $L$ homotopically commutes with fixed points, and in both $\cC$ and $\cD$ the tensor product homotopically commutes with fixed points, i.e., for $X,Y \in \cC^G$ cofibrant, the canonical map $\phy \colon X^H \ot Y^H \to (X\ot Y)^H$ is a weak equivalence in $\cC$ for all subgroup $H \leq G$. If the original adjunction $L \colon \cC \rla \cD \colon R$ is weak monoidal, then so is the induced adjunction $L \colon \cC^G \rla \cD^G \colon R$.
\end{enumerate}
\end{prop}

\begin{proof}
(1) The claim follows from the fact that the right adjoint $R$ commutes with fixed point functors, i.e., $R(Y^H) \cong (RY)^H$, as observed in \cite{Stephan16}*{Example~2.19}.

(2) Let $f \colon X \to RY$ be a map in $\cC^G$, for cofibrant $X \in \cC^G$ and fibrant $Y \in \cD^G$. We have equivalent conditions:
\begin{itemize}
    \item[] $f \colon X \to RY$ is a weak equivalence in $\cC^G$.
    \item[$\stackrel{\text{def}}{\iff}$] $f^H \colon X^H \to (RY)^H$ is a weak equivalence in $\cC$ for all subgroup $H \leq G$.
    \item[$\iff$] $f^H \colon X^H \to R(Y^H)$ is a weak equivalence in $\cC$ for all subgroup $H \leq G$, since $R$ commutes with fixed points.
    \item[$\iff$] The adjunct map $\tild{f^H} \colon L(X^H) \to Y^H$ is a weak equivalence in $\cD$ for all subgroup $H \leq G$, since $L \dashv R$ is a Quillen equivalence and $X^H$ is cofibrant in $\cC$, by Lemma~\ref{lem:TensorTrivial}\eqref{item:CofibrantFixed}.
    \item[$\iff$] $\tild{f}^H \colon (LX)^H \to Y^H$ is a weak equivalence in $\cD$ for all subgroup $H \leq G$, since $L$ homotopically commutes with fixed points.
    \item[$\stackrel{\text{def}}{\iff}$] $\tild{f} \colon LX \to Y$ is a weak equivalence in $\cD^G$.
\end{itemize}
(3) Consider the diagram of adjunctions
\[
\xymatrix{
\cC^{\cO_G^{\opp}} \ar@<-0.6ex>[d]_{L} \ar@<0.6ex>[r]^-{i^*} & \cC^G \ar@<0.6ex>[l]^-{i_*} \ar@<-0.6ex>[d]_{L} \\
\cD^{\cO_G^{\opp}} \ar@<-0.6ex>[u]_{R} \ar@<0.6ex>[r]^-{i^*} & \cD^G \ar@<0.6ex>[l]^-{i_*} \ar@<-0.6ex>[u]_{R} \\
}
\]
where the right adjoints commute (hence the left adjoints commute). By part~(2), the vertical adjunction between $G$-objects $L \colon \cC^G \rla \cD^G$ is a Quillen equivalence. By Proposition~\ref{pr:DiagramCat}, the vertical adjunction between presheaf categories $L \colon \cC^{\cO_G^{\opp}} \rla \cD^{\cO_G^{\opp}} \colon R$ is a Quillen equivalence. Hence the top row is a Quillen equivalence if and only if the bottom row is.

(4) \textbf{Lax monoidal:} The induced right adjoint $R \colon \cD^G \to \cC^G$ is lax monoidal, since the monoidal structures are of underlying objects, i.e., pointwise in the diagram category $\cC^G = \Fun(G,\cC)$.

\textbf{Oplax comultiplication:} Let $X,Y \in \cC^G$ be cofibrant. We want to show that the comultiplication map $\de \colon F(X \ot Y) \to F(X) \ot F(Y)$ is a weak equivalence in $\cD^G$, i.e., a weak equivalence on all fixed points. Note that $X \ot Y$ is cofibrant in $\cC^G$ by Proposition~\ref{pr:GMonoidal}, while fixed points $X^H$ and $Y^H$ are cofibrant in $\cC$, by Lemma~\ref{lem:TensorTrivial}\eqref{item:CofibrantFixed}. Consider the commutative diagram in $\cD$:
\[
\xymatrix @R-0.4pc {
F(X \ot Y)^H \ar[r]^-{\de^H}_-{\blue{\therefore \, \sim}} & \left( F(X) \ot F(Y) \right)^H \\
F \left( (X \ot Y)^H \right) \ar[u]^{\al}_{\sim} & F(X)^H \ot F(Y)^H \ar[u]_{\phy}^{\sim} \\
F \left( X^H \ot Y^H \right) \ar[u]^{F(\phy)}_{\sim} \ar[r]^-{\de}_-{\sim} & F(X^H) \ot F(Y^H). \ar[u]_{\al \ot \al}^{\sim} \\
}
\]
The map $F(\phy)$ on the left is a weak equivalence since $\phy \colon X^H \ot Y^H \to (X \ot Y)^H$ is a weak equivalence between cofibrant objects. The map $\al \ot \al$ on the right is a weak equivalence since it is a tensor product of weak equivalences between cofibrant objects. Thus the top map $\de^H$ is a weak equivalence.

\textbf{Unit condition:} Take a cofibrant replacement $\Triv(q) \colon \Triv(Q\unit_{\cC}) \to \Triv(\unit_{\cC}) = \unit_{\cC^G}$  in $\cC^G$ as in the proof of Proposition~\ref{pr:GMonoidal}. Applying the left adjoint $L$ and composing with the counit map yields:
\[
\xymatrix @C+0.5pc @R-0.2pc {
L \Triv(Q\unit_{\cC}) \ar[d]_{\cong} \ar[r]^-{L \Triv (q)} & L \Triv(\unit_{\cC}) \ar[d]_{\cong} \ar[r]^-{\ep} & \Triv(\unit_{\cD}) \ar@{=}[d] \\
\Triv(L Q\unit_{\cC}) \ar[r]^-{\Triv (Lq)} & \Triv( L\unit_{\cC}) \ar[r]^-{\Triv(\ep)} & \Triv(\unit_{\cD}). \\
}
\]
The bottom composite $\Triv(\ep \circ Lq)$ is a weak equivalence, by the %
unit condition for $L \dashv R$ and Lemma~\ref{lem:TrivialAction}.
\end{proof}

The condition $\phy \colon X^H \ot Y^H \ral{\sim} (X \ot Y)^H$ holds for the Cartesian product $\ot = \x$, since $\phy$ is then an isomorphism. The condition %
need not hold for other monoidal structures.

\begin{ex}
For any commutative ring $R$ and non-trivial group $G$, the tensor product in $\Ch_{\geq 0}(R)$ does \emph{not} homotopically commute with fixed points. Via the isomorphism of categories $\Mod{R}^G \cong \Mod{RG}$, consider the free $RG$-module $R \odot G/e = RG$, for which the map $\phy \colon (RG)^G \ot_R (RG)^G \to (RG \ot RG)^G$ is not an isomorphism. Indeed, the two sides are:
\[
(RG)^G \ot (RG)^G \cong R \ot R \cong R \qquad \text{versus} \qquad (RG \ot_R RG)^G \cong R[G \x G]^G \cong RG.
\]
Since the chain complex concentrated in degree $0$ $R[0] \in \Ch_{\geq 0}(R)$ is cofibrant, so is the chain complex $RG[0] \in \Ch_{\geq 0}(RG)$, by Lemma~\ref{lem:TrivialAction}. But the map $\phy \colon (RG[0])^G \ot (RG[0])^G \to (RG[0] \ot RG[0])^G$ is not a weak equivalence in $\Ch_{\geq 0}(R)$. 
The same argument works for the monoidal model category $s\Mod{R}$. %
\end{ex}

Nonetheless, Dold--Kan does induce a weak monoidal Quillen equivalence on $G$-objects, even though the assumptions of Proposition~\ref{pr:GObjects}\eqref{item:TensorFixed} are not met.

\begin{lem}\label{lem:DoldKanG}
For any commutative ring $R$, %
the induced adjunctions $N \colon s\Mod{R}^{G} \rla \Ch_{\geq 0}(R)^{G} \colon \Ga$ and $\Ga \colon \Ch_{\geq 0}(R)^{G} \rla s\Mod{R}^{G} \colon N$ are weak monoidal. %
\end{lem}

\begin{proof}
We prove the case of the adjunction $N \dashv \Ga$; the argument for the adjunction $\Ga \dashv N$ is similar. We want to show that for all cofibrant $A,B \in s\Mod{R}^{G} \cong s\Mod{RG}$, the oplax comultiplication map
\[
\delta = \AW \colon N(A \ot B) \to N(A) \ot N(B) 
\]
is a weak equivalence in $\Ch_{\geq 0}(R)^{G} \cong \Ch_{\geq 0}(RG)$, i.e., a weak equivalence on all fixed points. But $\AW$ is a chain homotopy equivalence for all $A,B \in s\Mod{RG}$. 
Since the fixed points functor $(-)^H \colon \Mod{RG} \to \Mod{R}$ is additive, it preserves chain homotopy equivalences, so that the map
\[
\AW^H \colon N(A \ot B)^H \ral{\simeq} \left( N(A) \ot N(B) \right)^H
\]
is a chain homotopy equivalence in $\Ch_{\geq 0}(R)$.
\end{proof}

\begin{ex}\label{ex:EquivariantDoldKan}
For a commutative ring $R$, the $q$-model structures on $s\Mod{R}$ and $\Ch_{\geq 0}(R)$ are %
cofibrantly generated, 
hence the projective model structures exist on the %
presheaf categories $s\Mod{R}^{\cO_G^{\opp}}$ and $\Ch_{\geq 0}(R)^{\cO_G^{\opp}}$, by Proposition~\ref{pr:DiagramModel}. The fine model structures also exist on the categories of $G$-objects $s\Mod{R}^{G}$ and $\Ch_{\geq 0}(R)^{G}$ \cite{Stephan16}*{Example~2.19}. Stephan showed that if the group $G$ is non-trivial, then the model category $\Ch_{\geq 0}(R)$ does \emph{not} satisfy Elmendorf's theorem, hence $s\Mod{R}$ doesn't either, by Proposition~\ref{pr:GObjects}. In the commutative diagram of Quillen pairs
\[
\xymatrix{
s\Mod{R}^{\cO_G^{\opp}} \ar@/_1.2ex/[d]_{N} \ar@{}[d]|{\blue{\sim}} \ar@/^1.2ex/[r]^-{i^*} \ar@{}[r]|{\red{\not\sim}} & s\Mod{R}^G \ar@/^1.2ex/[l]^-{i_*} \ar@{}[d]|{\blue{\sim}} \ar@/_1.2ex/[d]_{N} \\
\Ch_{\geq0}(R)^{\cO_G^{\opp}} \ar@/_1.2ex/[u]_{\Ga} \ar@/^1.2ex/[r]^-{i^*} \ar@{}[r]|{\red{\not\sim}} & \Ch_{\geq 0}(R)^G \ar@/^1.2ex/[l]^-{i_*} \ar@/_1.2ex/[u]_{\Ga} \\
}
\]
the vertical Quillen equivalences are two different versions of the equivariant Dold--Kan correspondence. 
Both are weak monoidal, by Lemma~\ref{lem:DoldKanG}.
\end{ex}

\section{Bousfield localizations}

\begin{prop}\label{pr:CompositeMonoidal}
Let $\cU$, $\cV$, and $\cW$ be %
monoidal model categories. A composite of weak monoidal Quillen adjunctions
\[
\xymatrix{
\cW \ar@/^0.5pc/[r]^-{F} \ar@{}[r]|{\bot} & \cV \ar@/^0.5pc/[l]^-{G} \ar@/^0.5pc/[r]^-{F'} \ar@{}[r]|{\bot} & \cU \ar@/^0.5pc/[l]^-{G'} \\
}
\]
is also a weak monoidal Quillen adjunction.
\end{prop}

\begin{proof}
A composite of Quillen adjunctions is a Quillen adjunctions. The composite $GG' \colon \cU \to \cW$ of lax monoidal functors $G$ and $G'$ is also lax monoidal.

\textbf{Oplax comultiplication:} Let $x,y \in \cW$ be cofibrant objects. By assumption on $F$, the oplax structure map
$\de \colon F(x \ot y) \ral{\sim} F(x) \ot F(y)$ 
is a weak equivalence in $\cV$. Moreover, both $F(x \ot y)$ and $F(x) \ot F(y)$ are cofibrant by Lemma~\ref{lem:cofibcoten} and since the left Quillen functor $F$ preserves cofibrant objects. Applying the left Quillen functor $F'$ yields a weak equivalence $F(\de)$, and the composite
\[
\xymatrix{
F'F(x \ot y) \ar[r]^-{F(\de)}_-{\sim} & F'(Fx \ot Fy) \ar[r]^-{\de'}_-{\sim} & F'F(x) \ot F'F(y) \\ 
}
\]
is also a weak equivalence.

\textbf{Unit condition:} Let $q \colon Q\unit_{\cW} \ral{\sim} \unit_{\cW}$ be a cofibrant replacement of the tensor unit in $\cW$. By assumption on $F$, the composite 
\[
\xymatrix{
F(Q \unit_{\cW}) \ar[r]^-{F(q)} & F(\unit_{\cW}) \ar[r]^-{\ep} & \unit_{\cV} \\
}
\]
is a weak equivalence in $\cV$, hence a cofibrant replacement of $\unit_{\cV}$. By assumption on $F'$, the composite 
\[
\xymatrix{
F'F(Q \unit_{\cW}) \ar[r]^-{F'F(q)} & F'F(\unit_{\cW}) \ar[r]^-{F'(\ep)} & F'(\unit_{\cV}) \ar[r]^-{\ep'} & \unit_{\cU} \\
}
\]
is a weak equivalence in $\cU$.
\end{proof}

\begin{defn}\cite{Hirschhorn03}*{Definition~3.3.1}
    Let $\cC$ be a model category with model structure $(W,F,C)$. 
    \begin{enumerate}
        \item A \Def{left Bousfield localization} of $\cC$ is another model structure $(W',F',C')$ on $\cC$ satisfying $W \subseteq W'$ and with the same cofibrations $C = C'$.
        \item A \Def{right Bousfield localization} (also called a \emph{colocalization}) of $\cC$ is another model structure $(W',F',C')$ on $\cC$ satisfying $W \subseteq W'$ and with the same fibrations $F = F'$.
    \end{enumerate}
\end{defn}

\begin{defn}
Let $\cV$ be a monoidal model category. A (left or right) Bousfield localization $\cV'$ of $\cV$ is called a \Def{monoidal Bousfield localization} if $\cV'$ is also a monoidal model category, with the same monoidal structure as $\cV$.

Unlike \cite{White22}*{Definition~5.10}, we do not include the condition that all cofibrant objects of $\cV'$ be flat.
\end{defn}

\begin{prop}\label{pr:MonoidalBousfield}
Let $\cV$ be a monoidal model category. 
	\begin{enumerate}
		\item \label{item:LeftBousfield} If $\cV_{\loc}$ is a monoidal (left) Bousfield localization of $\cV$, then the Quillen pair $\id \colon \cV \rla \cV_{\loc} \colon \id$ is a strong monoidal Quillen pair.
		\item \label{item:RightBousfield} If $\cV_{\coloc}$ is a monoidal right Bousfield localization of $\cV$, then the Quillen pair $\id \colon \cV_{\coloc} \rla \cV \colon \id$ is a strong monoidal Quillen pair if and only if a cofibrant replacement $Q\unit$ of the tensor unit of $\cV$ is also cofibrant in $\cV_{\coloc}$.
	\end{enumerate}
\end{prop}

\begin{proof}
In both cases, the left adjoint $\id$ is strong monoidal, so that Remark~\ref{rem:StrongMonoidal} applies.

\Eqref{item:LeftBousfield} Let $q \colon Q\unit \ral{\sim} \unit$ be a cofibrant replacement of the $\unit$ in $\cV$. Applying the identity functor yields the same map, which is also a weak equivalence in $\cV_{\loc}$.

\Eqref{item:RightBousfield} Let $q_c \colon Q_c \unit \ral{\sim_c} \unit$ be a cofibrant replacement of $\unit$ in $\cV_{\coloc}$. Since $Q_c \unit$ is also cofibrant in $\cV$, there is a factorization
\[
\xymatrix{
Q_c \ar@{-->}[dr]_{f} \ar[rr]^-{q_c}_{\sim_c} \unit & & \unit \\
& Q\unit \ar@{->>}[ur]_{q}^{\sim} & \\
}
\]
where $q \colon Q\unit \ral{\sim} \unit$ is an acyclic fibration from a cofibrant object. The map $q_c$ is a weak equivalence in $\cV$ if and only if the map $f \colon Q_c \unit \to Q \unit$ is, which holds if and only if $Q\unit$ is cofibrant in $\cV_{\coloc}$ \cite{Hirschhorn03}*{Proposition~3.3.15}.
\end{proof}

Combining Propositions~\ref{pr:CompositeMonoidal} and \ref{pr:MonoidalBousfield} yields the following.

\begin{cor}\label{cor:MonoidalBousfield}
Let $F \colon \cW \rla \cV \colon G$ be a weak monoidal Quillen adjunction between monoidal model categories. 
\begin{enumerate}
	\item If $\cV_{\loc}$ is a monoidal (left) Bousfield localization of $\cV$, then $F \colon \cW \rla \cV_{\loc} \colon G$ is a weak monoidal Quillen adjunction.
	\item If $\cW_{\coloc}$ is a monoidal right Bousfield localization of $\cW$ such that a cofibrant replacement $Q\unit_{\cW}$ of the tensor unit of $\cW$ is also cofibrant in $\cW_{\coloc}$, then $F \colon \cW_{\coloc} \rla \cV \colon G$ is a weak monoidal Quillen adjunction.
\end{enumerate}
\end{cor}

When are Bousfield localizations monoidal? Let us recall some facts observed for instance in \cite{White22}*{\S 5.4} or \cite{Lawson22}*{\S 12.2}.

\begin{lem}
Let $\cV$ be a monoidal model category.
\begin{enumerate}
	\item \label{item:LeftBousfieldMonoidal} A (left) Bousfield localization $\cV_{\loc}$ of $\cV$ is monoidal if and only if for all cofibrations $i$ and $j$ in $\cV$ with at least one of them a weak equivalence in $\cV_{\loc}$, the pushout-product $i \square j$ is a weak equivalence in $\cV_{\loc}$.
	\item \label{item:RightBousfieldMonoidal} A right Bousfield localization $\cV_{\coloc}$ of $\cV$ is monoidal if and only if the following conditions hold:
	\begin{itemize}
		\item[(i)] For all cofibrations $i$ and $j$ in $\cV_{\coloc}$, the pushout-product $i \square j$ is a cofibration in $\cV_{\coloc}$.
		\item[(ii)] For a cofibrant replacement $Q_c \unit \ral{\sim_c} Q\unit$ in $\cV_{\coloc}$ (where $q \colon Q\unit \ral{\sim} \unit$ is a cofibrant replacement in $\cV$) and every cofibrant object $x \in \cV_{\coloc}$, the map $x \ot Q_c \unit \to x \ot Q\unit$ is a weak equivalence in $\cV_{\coloc}$.
	\end{itemize}		
\end{enumerate}
\end{lem}

\begin{proof}
In both cases, the remaining conditions hold automatically.
\end{proof}

For conditions under which a Bousfield localization is monoidal, see \cite{White22}*{Theorems~5.11, 5.12} and \cite{Barwick10}*{Proposition~4.47}. 
For results focusing on stable model categories, see \cite{BarnesR14}*{\S 6,7} and \cites{PolW20,PolW21}.

\begin{ex}
    Let $\cV$ be a monoidal category admitting two model structures $(W_q, F_q, C_q)$ and $(W_h, F_h, C_h)$ that satisfy $W_h \subseteq W_q$ and $F_h \subseteq F_q$. If $\cV$ is a monoidal model category with respect to both model structures, then $\cV$ is also a monoidal model category with respect to the mixed model structure $(W_q, F_h, C_m)$ \cite{Cole06}*{Proposition~6.6}. This makes $\cV_m$ a monoidal right Bousfield localization of $\cV_h$.
\end{ex}

\begin{ex}
    By Example~\ref{ex:MonoidalModelCats}~\Eqref{item:sMod}\Eqref{item:Chgeq0}, the mixed model structures on chain complexes $\Ch_{\geq 0}(R)_m$ and simplicial modules $s\Mod{R,m}$ are monoidal right Bousfield localizations of the Hurewicz model structures $\Ch_{\geq 0}(R)_h$ and $s\Mod{R,h}$ respectively. %
    By Example~\ref{ex:DK-WeakMonoidal} %
    and Corollary~\ref{cor:MonoidalBousfield}, the adjunctions
    \[
    N \colon s\Mod{R,m} \rla \Ch_{\geq 0}(R)_h \colon \Ga \quad \text{and} \quad \Ga \colon \Ch_{\geq 0}(R)_m \rla s\Mod{R,h} \colon N
    \]
    are weak monoidal Quillen adjunctions.
\end{ex}

\vspace*{8pt}

\end{document}